\documentclass[a4paper]{amsart}
\usepackage{amsmath}
\usepackage{amssymb}
\usepackage{amsthm,amsxtra}
\usepackage{cite}
\usepackage{enumerate}
\usepackage{enumitem}
\usepackage[mathscr]{eucal}
\usepackage{adjustbox}
\usepackage{multirow}
\usepackage{multicol}
\usepackage{float}
\usepackage{tikz-cd}
\usepackage{tkz-graph}
\usepackage{tkz-berge}
\usetikzlibrary{positioning,arrows,shapes.geometric,trees}
\usepackage{graphicx}
\usepackage{mathtools}
\usetikzlibrary{positioning}
\usetikzlibrary{decorations.text}
\usepackage{refcount}
\newtheorem{Def}{Definition}[section]
\newtheorem{Th}{Theorem}[section]
\newtheorem{Ex}{Example}[section]
\newtheorem{Lemma}{Lemma}[section]

\newtheorem{Cor}{Corollary}[section]

{}
\newtheorem{Prob}{Problem}[section]
\DeclareMathOperator{\cov}{\sf cov}
\begin{document}
\title[Further investigations on certain star selection principles]{Further investigations on certain star selection principles}

\author[ D. Chandra, N. Alam ]{ Debraj Chandra$^*$, Nur Alam$^*$ }
\newcommand{\acr}{\newline\indent}
\address{\llap{*\,}Department of Mathematics, University of Gour Banga, Malda-732103, West Bengal, India}
\email{debrajchandra1986@gmail.com, nurrejwana@gmail.com}

\thanks{ The second author
is thankful to University Grants Commission (UGC), New Delhi-110002, India for granting UGC-NET Junior Research Fellowship (1173/(CSIR-UGC NET JUNE 2017)) during the tenure of which this work was done.}


\maketitle

\begin{abstract}
We consider certain star versions of the Menger, Hurewicz and Rothberger properties. Few important observations concerning these properties are presented, which have not been investigated in earlier works. A variety of investigations is performed using Alster covers and critical cardinalities $\mathfrak{d}$, $\mathfrak{b}$ and $\cov(\mathcal{M})$. Our study explores further ramifications on the extent and Alexandroff duplicate. In the process we present investigations on the star versions of the Rothberger property and compare with similar prior observations of the star versions of the Menger and Hurewicz properties. We sketch few tables that interpret (mainly preservation-kind of) properties of the star selection principles obtained so far. We also present implication diagrams to explicate the interplay between the star selection principles.
\end{abstract}
\smallskip

\noindent{\bf\keywordsname{}:} {star-Menger, star-Hurewicz, star-Rothberger, star-K-Menger, star-K-Hurewicz, strongly star-Menger, strongly star-Hurewicz, strongly star-Rothberger.}

\section{Introduction}
The study of star selection principles is well known. For recent development of star selection principles one can consult the papers \cite{dcna21,dcna22,dcna22-1,LjSM,sHR,survey,survey1,SVM,SSSP} and references therein. Let $\mathcal{A}$ and $\mathcal{B}$ be collections of sets. In \cite{coc1} (see also \cite{coc2}), Scheepers began the systematic study of selection principles by introducing the following.

\noindent $S_1(\mathcal{A},\mathcal{B})$: For each sequence $(\mathcal{U}_n)$ of elements of $\mathcal{A}$ there exists a sequence $(V_n)$ such that for each $n$ $V_n\in\mathcal{U}_n$ and $\{V_n : n\in\mathbb{N}\}\in\mathcal{B}$.

\noindent $S_{fin}(\mathcal{A},\mathcal{B})$: For each sequence $(\mathcal{U}_n)$ of elements of $\mathcal{A}$ there exists a sequence $(\mathcal{V}_n)$ such that for each $n$ $\mathcal{V}_n$ is a finite subset of $\mathcal{U}_n$ and $\cup_{n\in\mathbb{N}}\mathcal{V}_n\in\mathcal{B}$.

\noindent $U_{fin}(\mathcal{A},\mathcal{B})$: For each sequence $(\mathcal{U}_n)$ of elements of $\mathcal{A}$ there exists a sequence $(\mathcal{V}_n)$ such that for each $n$ $\mathcal{V}_n$ is a finite subset of $\mathcal{U}_n$ and $\{\cup\mathcal{V}_n : n\in\mathbb{N}\}\in\mathcal{B}$ or $\cup\mathcal{V}_n=X$ for some $n$.

For a subset $A$ of a space $X$ and a collection $\mathcal{P}$
of subsets of $X$, $St(A,\mathcal{P})$ denotes the star of $A$
with respect to $\mathcal{P}$, that is the set
$\cup\{B\in\mathcal{P} : A\cap B\neq\emptyset\}$. For $A=\{x\}$,
$x\in X$, we write $St(x,\mathcal{P})$ instead of
$St(\{x\},\mathcal{P})$ \cite{Engelking}. In \cite{LjSM}, Ko\v{c}inac introduced the next four selection principles in the following way.

\noindent $S_1^*(\mathcal{A},\mathcal{B})$: For each sequence $(\mathcal{U}_n)$ of elements of $\mathcal{A}$ there exists a sequence $(V_n)$ such that for each $n$ $V_n\in\mathcal{U}_n$ and $\{St(V_n,\mathcal{U}_n) : n\in\mathbb{N}\}\in\mathcal{B}$.

\noindent $S_{fin}^*(\mathcal{A},\mathcal{B})$: For each sequence
$(\mathcal{U}_n)$ of elements of $\mathcal{A}$ there exists a
sequence $(\mathcal{V}_n)$ such that for each $n$
$\mathcal{V}_n$ is a finite subset of $\mathcal{U}_n$ and
$\cup_{n\in\mathbb{N}}\{St(V,\mathcal{U}_n) :
V\in\mathcal{V}_n\}\in\mathcal{B}$.

\noindent $U_{fin}^*(\mathcal{A},\mathcal{B})$: For each sequence
$(\mathcal{U}_n)$ of elements of $\mathcal{A}$ there exists a
sequence $(\mathcal{V}_n)$ such that for each $n$
$\mathcal{V}_n$ is a finite subset of $\mathcal{U}_n$ and
$\{St(\cup\mathcal{V}_n,\mathcal{U}_n) :
n\in\mathbb{N}\}\in\mathcal{B}$ or there is some $n$ such that
$St(\cup\mathcal{V}_n,\mathcal{U}_n)=X$.

Let $\mathcal{K}$ be a collection of subsets of $X$. Then we
define

\noindent $S_{\mathcal{K}}^*(\mathcal{A},\mathcal{B})$: For each sequence $(\mathcal{U}_n)$ of elements of $\mathcal{A}$ there exists a sequence $(K_n)$ of elements of $\mathcal{K}$ such that
$\{St(K_n,\mathcal{U}_n) : n\in\mathbb{N}\}\in\mathcal{B}$.
If $\mathcal{K}$ is the collection of one-point (respectively, finite, compact) subsets of $X$, then we use $SS_1^*(\mathcal{A},\mathcal{B})$ (respectively, $SS_{fin}^*(\mathcal{A},\mathcal{B})$, $SS_{comp}^*(\mathcal{A},\mathcal{B})$) instead of $S_{\mathcal{K}}^*(\mathcal{A},\mathcal{B})$.

In this article we study star variants of the Menger, Hurewicz and Rothberger properties, namely star-Menger (-Hurewicz, -Rothberger) property, star-K-Menger (-Hurewicz) property and strongly star-Menger (-Hurewicz, -Rothberger) property. We present interesting observations on these properties which have not been investigated before. In particular certain investigations are carried out using Alster covers and critical cardinalities ($\mathfrak{d}$, $\mathfrak{b}$ and $\cov(\mathcal{M})$), and we also obtain few interesting observations on the extent and Alexandroff duplicate. Besides, in \cite[Example 2.4]{sKM}, it was observed that there exists a $T_1$ star-Menger space which is not star-K-Menger. We show that the space $X$ as in \cite[Example 2.4]{sKM} is indeed star-K-Menger. Our investigation contradicts the following result of Ko\v{c}inac \cite{LjSM} (given without proof).
\begin{Th}[{\!\cite[Theorem 2.13]{LjSM}}]
If $X$ is star-Rothberger and $Y$ is compact, then $X\times Y$ is star-Rothberger.
\end{Th}

We also give answers to the following problems posed in \cite{rsM,rssM}.
\begin{enumerate}[wide=0pt,label={\upshape(\arabic*)},leftmargin=*]
\item Is there a space $X$ such that $A(X)$ is star-Menger, but $X$ is not star-Menger?
\item Is there a space $X$ such that $A(X)$ is strongly star-Menger, but $X$ is not strongly star-Menger?
\end{enumerate}

Finally the interrelationship between star selection principles are outlined into implication diagrams (Figures~\ref{dig1}-\ref{dig4}).

\section{Definitions and terminologies}
By a space we always mean a topological space. For undefined notions and terminologies, see \cite{Engelking}. Let $\mathcal O$ denote the collection of all open covers of $X$. We use the following notions of open covers throughout the paper from \cite{coc1,coc2,cocVII,cocVIII}.
\begin{enumerate}[leftmargin=*]
  \item[$\Gamma$:] The collection of all $\gamma$-covers of $X$. An open cover $\mathcal{U}$ of $X$ is said to be a $\gamma$-cover if it is infinite and each element of $X$ does not belong to at most finitely many members of $\mathcal{U}$ \cite{coc1,coc2}.
  \item[$\Omega$:] The collection of all $\omega$-covers of $X$. An open cover $\mathcal{U}$ of $X$ is said to be an $\omega$-cover if $X$ does not belong to $\mathcal{U}$ and for each finite subset $F$ of $X$ there is a set $U\in\mathcal{U}$ such that $F\subseteq U$ \cite{coc1,coc2}.
  \item[$\Lambda$:] The collection of all large covers of $X$. An open cover $\mathcal{U}$ of $X$ is said to be a large cover if for each $x\in X$, the set $\{U\in\mathcal{U} : x\in U\}$ is infinite \cite{coc1,coc2}.
  \item[$\mathcal{O}^{gp}$:] The collection of all groupable covers of $X$. An open cover $\mathcal{U}$ of $X$ is said to be groupable if it can be expressed as a countable union of finite, pairwise disjoint subfamilies $\mathcal{U}_n$, $n\in\mathbb{N}$, such that each $x\in X$ belongs to $\cup\mathcal{U}_n$ for all but finitely many $n$ \cite{cocVII}.
  \item[$\mathcal{O}^{wgp}$:] The collection of all weakly groupable covers of $X$. An open cover $\mathcal{U}$ of $X$ is said to be weakly groupable if it can be expressed as a countable union of finite, pairwise disjoint subfamilies $\mathcal{U}_n$, $n\in\mathbb{N}$, such that for each finite set $F\subseteq X$ we have $F\subseteq\cup\mathcal{U}_n$ for some $n$ \cite{cocVIII}.
\end{enumerate}

Note that $\Gamma\subseteq\Omega\subseteq\Lambda\subseteq\mathcal{O}$.

A space $X$ is said to have the Menger (respectively, Hurewicz, Rothberger) property if it satisfies the selection hypothesis $S_{fin}(\mathcal{O},\mathcal{O})$ (respectively, $U_{fin}(\mathcal{O},\Gamma)$, $S_1(\mathcal{O},\mathcal{O})$) \cite{coc1,coc2}. A space $X$ is said to have the (1) star-Menger property, (2) star-Hurewicz property, (3) star-Rothberger property, (4) star-K-Menger property, (5) star-K-Hurewicz property, (6) strongly star-Menger property, (7) strongly star-Hurewicz property and (8) strongly star-Rothberger property if it satisfies the selection hypothesis (1) $S_{fin}^*(\mathcal{O},\mathcal{O})$, (2) $U_{fin}^*(\mathcal{O},\Gamma)$, (3) $S_1^*(\mathcal{O},\mathcal{O})$, (4) $S_{comp}^*(\mathcal{O},\mathcal{O})$, (5) $S_{comp}^*(\mathcal{O},\Gamma)$, (6) $SS_{fin}^*(\mathcal{O},\mathcal{O})$, (7) $SS_{fin}^*(\mathcal{O},\Gamma)$ and (8) $SS_1^*(\mathcal{O},\mathcal{O})$ respectively \cite{LjSM,sHR}
(see also \cite{rsM,rssM,sKH,sKM,SR}).

A space $X$ is said to be starcompact (respectively, star-Lindel\"{o}f) if for every open cover $\mathcal{U}$ of $X$ there exists a finite (respectively, countable) $\mathcal{V}\subseteq\mathcal{U}$ such that $St(\cup\mathcal{V},\mathcal{U})=X$ \cite{LjSM,sCP}. $X$ is said to be strongly starcompact (respectively, K-starcompact, strongly star-Lindel\"{o}f, star-L-Lindel\"{o}f) if for every open cover $\mathcal{U}$ of $X$ there exists a finite (respectively, compact, countable, Lindel\"{o}f) set $K\subseteq X$ such that $St(K,\mathcal{U})=X$ \cite{LjSM,sCP,sKM}. It is to be noted that every starcompact (respectively, K-starcompact, strongly starcompact) space is star-Hurewicz (respectively, star-K-Hurewicz, strongly star-Hurewicz) and every star-Menger (respectively, star-K-Menger, strongly star-Menger) space is star-Lindel\"{o}f (respectively, star-L-Lindel\"{o}f, strongly star-Lindel\"{o}f).

The eventual dominance relation $\leq^*$ on the Baire space $\mathbb{N}^\mathbb{N}$ is defined by $f\leq^*g$ if and only if $f(n)\leq g(n)$ for all but finitely many $n$. A subset $A$ of $\mathbb{N}^\mathbb{N}$ is said to be dominating if for each $g\in\mathbb{N}^\mathbb{N}$ there exists a $f\in A$ such that $g\leq^* f$. A subset $A$ of $\mathbb{N}^\mathbb{N}$ is said to be bounded if there is a $g\in\mathbb{N}^\mathbb{N}$ such that $f\leq^*g$ for all $f\in A$. Moreover a set $A\subseteq\mathbb{N}^\mathbb{N}$ is said to be guessed by $g\in\mathbb{N}^\mathbb{N}$ if $\{n\in\mathbb{N} : f(n)=g(n)\}$ is infinite for all $f\in A$. The minimum cardinality of a dominating subset of $\mathbb{N}^\mathbb{N}$ is denoted by $\mathfrak{d}$, and the minimum cardinality of an unbounded subset of $\mathbb{N}^\mathbb{N}$ is denoted by $\mathfrak{b}$. Let $\cov(\mathcal{M})$ be the minimum cardinality of a family of meager subsets of the set of reals $\mathbb{R}$ that covers $\mathbb{R}$. In \cite{CAMC} (see also \cite[Theorem 2.4.1]{TBHJ}), $\cov(\mathcal{M})$ is described as the minimum cardinality of a subset $F\subseteq\mathbb{N}^\mathbb{N}$ such that for every $g\in\mathbb{N}^\mathbb{N}$ there is a $f\in F$ such that $f(n)\neq g(n)$ for all but finitely many $n$. Thus we can say that if $F\subseteq\mathbb{N}^\mathbb{N}$ and $|F|<\cov(\mathcal{M})$, then $F$ can be guessed by a $g\in\mathbb{N}^\mathbb{N}$. Let $\mathfrak{c}$ be the cardinality of the set of reals. It is well known that $\omega_1\leq\mathfrak{b}\leq\mathfrak{d}\leq\mathfrak{c}$ and $\omega_1\leq\cov(\mathcal{M})\leq\mathfrak{d}$. For any cardinal $\kappa$, $\kappa^+$ denotes the smallest cardinal greater than $\kappa$.

A family $\mathcal{A}\subseteq P(\mathbb{N})$ is said to be an almost disjoint family if each $A\in\mathcal{A}$ is infinite and for any two distinct elements $B,C\in\mathcal{A}$, $|B\cap C|<\omega$. Also $\mathcal{A}$ is said to be a maximal almost disjoint (in short, MAD) family if $\mathcal{A}$ is not contained in any larger almost disjoint family. For an almost disjoint family $\mathcal{A}$, let $\Psi(\mathcal{A})=\mathcal{A}\cup\mathbb{N}$ be the Isbell-Mr\'{o}wka space (or, $\Psi$-space) (see \cite{Mrowka}). It is well known that $\Psi(\mathcal{A})$ is pseudocompact if and only if $\mathcal{A}$ is a maximal almost disjoint family. In general, when talking about Isbell-Mr\'{o}wka space we do not require almost disjoint family to be maximal or the space to be pseudocompact. The Alexandroff duplicate $A(X)$ of a space $X$ (see \cite{AD,Engelking}) is defined as follows. $A(X)=X\times\{0,1\}$; each point of $X\times\{1\}$ is isolated and a basic neighbourhood of $\langle x,0\rangle\in X\times\{0\}$ is a set of the form $(U\times\{0\})\cup((U\times\{1\})\setminus\{\langle x,1\rangle\})$, where $U$ is a neighbourhood of $x$ in $X$. For a Tychonoff space $X$, $\beta X$ denotes the \v{C}ech-Stone compactification of $X$. For any two collections $\mathcal{A}$ and $\mathcal{B}$ of subsets of a space $X$, we denote $\mathcal{A}\wedge\mathcal{B}$ by the set $\{A\cap B : A\in\mathcal{A}, B\in\mathcal{B}\}$.

\section{Star versions of the Menger, Hurewicz and Rothberger properties}
\subsection{Interrelationships}
We begin with a basic implication diagram (Figure~\ref{dig1}) for star selection principles.
\begin{figure}[h!]
\begin{adjustbox}{max width=0.7\textwidth,max
height=0.7\textheight,keepaspectratio,center}
\begin{tikzcd}[column sep=3ex,row sep=6ex]
U_{fin}^*(\mathcal{O},\Gamma)\arrow[rr] &&
S_{fin}^*(\mathcal{O},\mathcal{O}) &&
S_1^*(\mathcal{O},\mathcal{O})\arrow[ll]
\\
SS_{comp}^*(\mathcal{O},\Gamma)\arrow[rr]\arrow[u]&&
SS_{comp}^*(\mathcal{O},\mathcal{O})\arrow[u]&&
\\
SS_{fin}^*(\mathcal{O},\Gamma)\arrow[rr]\arrow[u]&&
SS_{fin}^*(\mathcal{O},\mathcal{O})\arrow[u]&&
SS_1^*(\mathcal{O},\mathcal{O})\arrow[ll]\arrow[uu]
\\
 U_{fin}(\mathcal{O},\Gamma) \arrow[rr]\arrow[u]&&
S_{fin}(\mathcal{O},\mathcal{O})\arrow[u]&&
S_1(\mathcal{O},\mathcal{O})\arrow[ll]\arrow[u]
\end{tikzcd}
\end{adjustbox}
\caption{Star variations of the Menger, Hurewicz and Rothberger properties}
\label{dig1}
\end{figure}

The reverse implications for the first and second columns of Figure~\ref{dig1} do not hold in general, see \cite{sHR,sKH,RSSH,RSH} and \cite{LjSM,rsM,rsM-II,rssM,sKM,RSKM} respectively for details.
Note that the space as in \cite[Example 2.4]{sKM} was claimed to be star-Menger but not star-K-Menger. Such space is indeed star-K-Menger. For convenience we present the corrected version of this example.
\begin{Ex}\rm
\label{E1}
Let $X$ be the space given by $X=[0,\omega_1)\cup D$, where $D=\{d_\alpha : \alpha<\omega_1\}$ is a set with cardinality $\omega_1$ and the topology on $X$ is defined as follows. The space $[0,\omega_1)$ has the usual order topology and it is an open subspace of $X$, a basic neighbourhood of a point $d_\alpha\in D$ takes the form $O_\beta(d_\alpha)=\{d_\alpha\}\cup(\beta,\omega_1),\;
\text{where}\;\beta<\omega_1.$ Then $X$ is star-K-Menger.
\end{Ex}
\begin{proof}
Since every K-starcompact space is star-K-Menger, it is enough to show that $X$ is K-starcompact. We first show that for each $\alpha<\omega_1$, $\omega_1\cup\{d_\alpha\}$ is a compact subset of $X$. Let $\alpha<\omega_1$ be fixed and $\mathcal{U}$ be a cover of $\omega_1\cup\{d_\alpha\}$ by open sets in $X$. Then we get a $U\in\mathcal{U}$ such that $O_\alpha(d_\alpha)\subseteq U$. Since $[0,\alpha]$ is compact, we have a finite subset $\mathcal{V}\subseteq\mathcal{U}$ such that $[0,\alpha]\subseteq\cup\mathcal{V}$. This gives us $$\omega_1\cup\{d_\alpha\}\subseteq\cup(\{U\}\cup\mathcal{V})$$ and consequently $\omega_1\cup\{d_\alpha\}$ is compact. It follows that for each $\alpha<\omega_1$, $\omega_1\cup\{d_\alpha\}$ is a compact subset of $X$. To show that $X$ is K-starcompact we pick an open cover $\mathcal{W}$ of $X$. Let $\beta<\omega_1$ be fixed and $K=\omega_1\cup\{d_\beta\}$. Since $K$ intersects every member of $\mathcal{W}$, it follows that $X=St(K,\mathcal{W})$. Thus $X$ is K-starcompact and this completes the proof.
\end{proof}

It is also interesting to observe that the reverse implications for the third column need not hold. Below we present two examples of Tychonoff spaces one of which is strongly star-Rothberger but not Rothberger and other is star-Rothberger but not strongly star-Rothberger. Under the assumption $\omega_1<\cov(\mathcal{M})$, the space $\Psi(\mathcal{A})$ with $|\mathcal{A}|=\omega_1$ is strongly star-Rothberger (Theorem~\ref{TN56}). Since $\mathcal{A}$ is a closed and discrete subset of $\Psi(\mathcal{A})$ with $|\mathcal{A}|=\omega_1$, $\Psi(\mathcal{A})$ is not Lindel\"{o}f (hence not Rothberger). Next let $\mathcal{A}\subseteq P(\mathbb{N})$ be an almost disjoint family with $|\mathcal{A}|=\mathfrak{d}$ such that for each function $A\mapsto f_A$ from $\mathcal{A}$ to $\mathbb{N}^\mathbb{N}$ there are elements $A_1,A_2,\dotsc\in\mathcal{A}$ such that for each $A\in\mathcal{A}$ there exists a $n$ such that $(A\setminus f_A(n))\cap(A_n\setminus f_{A_n}(n))\neq\emptyset$. The space $\Psi(\mathcal{A})$ is star-Rothberger (Theorem~\ref{TN59}) but not strongly star-Menger (Theorem~\ref{TN9}) and hence not strongly star-Rothberger.

On a similar note, none of the implications for the rows of Figure~\ref{dig1} are reversible. To see this, we reproduce classical examples of spaces (from \cite{coc2}) that were used to distinguish the Hurewicz, Menger and Rothberger properties. Such examples serve the purpose as the considered spaces are paracompact Hausdorff and in the context of paracompact Hausdorff spaces, all the star variations coincide with the corresponding classical selection principles (see \cite[Theorem 2.8]{LjSM}, \cite[Theorem 2.9]{sKH} and \cite[Theorem 2.9]{LjSM}). Every Sierpi\'{n}ski set (i.e. an uncountable subset of reals which has countable intersection with every set of Lebesgue measure zero) is Hurewicz and hence Menger but not Rothberger. Also every Lusin set (i.e. an uncountable subset of reals whose intersection with every first category set of reals is countable) is Rothberger and hence Menger but not Hurewicz. Counter examples for such reverse implications of Figure~\ref{dig1} can also be constructed using $\Psi$-spaces.

\begin{Ex}\rm
\hfill
\begin{enumerate}[wide=0pt,label={\upshape(\bf\arabic*)},leftmargin=*]
\item Assume that $\cov(\mathcal{M})<\mathfrak{b}$. There exists an almost disjoint family $\mathcal{A}\subseteq P(\mathbb{N})$ with cardinality $\cov(\mathcal{M})$ such that $\Psi(\mathcal{A})$ is not star-Rothberger (see \cite[Example 5]{SCPP}) and hence not strongly star-Rothberger. Since $|\mathcal{A}|<\mathfrak{b}$, by Theorem~\ref{TN55}, $\Psi(\mathcal{A})$ is strongly star-Hurewicz (hence strongly star-Menger, star-K-Menger, star-K-Hurewicz, star-Menger and star-Hurewicz).

  \item Next assume that $\mathfrak{b}=\aleph_1<\cov(\mathcal{M})$. Let $\mathcal{A}\subseteq P(\mathbb{N})$ be an almost disjoint family with $|\mathcal{A}|=\mathfrak{b}$. By \cite[Theorem 2.4]{CASC}, $\Psi(\mathcal{A})$ is not star-Hurewicz (hence not star-K-Hurewicz and not strongly star-Hurewicz) but strongly star-Rothberger (hence strongly star-Menger, star-K-Menger, star-Rothberger and star-Menger) by Theorem~\ref{TN56}.
\end{enumerate}
\end{Ex}
We mention the following combinatorial characterizations for $\Psi$-spaces.
\begin{Th}[{\!\cite[Theorem 2.1]{CASC}}]
\label{TN57}
The following assertions are equivalent.
\begin{enumerate}[wide=0pt,label={\upshape(\arabic*)},leftmargin=*]
  \item $\Psi(\mathcal{A})$ is star-Menger.
  \item For each function $A\mapsto f_A$ from $\mathcal{A}$ to $\mathbb{N}^\mathbb{N}$ there are finite sets $\mathcal{F}_1,\mathcal{F}_2,\dotsc\subseteq\mathcal{A}$ such that for each $A\in\mathcal{A}$ there exists a $n$ such that $(A\setminus f_A(n))\cap(\cup_{B\in\mathcal{F}_n}(B\setminus f_B(n)))\neq\emptyset$.
\end{enumerate}
\end{Th}

\begin{Th}[{\!\cite[Theorem 2.2]{CASC}}]
\label{TN58}
The following assertions are equivalent.
\begin{enumerate}[wide=0pt,label={\upshape(\arabic*)},leftmargin=*]
  \item $\Psi(\mathcal{A})$ is star-Hurewicz.
  \item For each function $A\mapsto f_A$ from $\mathcal{A}$ to $\mathbb{N}^\mathbb{N}$ there are finite sets $\mathcal{F}_1,\mathcal{F}_2,\dotsc\subseteq\mathcal{A}$ such that for each $A\in\mathcal{A}$, $(A\setminus f_A(n))\cap(\cup_{B\in\mathcal{F}_n}(B\setminus f_B(n)))\neq\emptyset$ for all but finitely many $n$.
\end{enumerate}
\end{Th}

\begin{Th}[{\!\cite[Theorem 4.3]{CASC}}]
\label{TN59}
The following assertions are equivalent.
\begin{enumerate}[wide=0pt,label={\upshape(\arabic*)},leftmargin=*]
  \item $\Psi(\mathcal{A})$ is star-Rothberger.
  \item For each function $A\mapsto f_A$ from $\mathcal{A}$ to $\mathbb{N}^\mathbb{N}$ there are elements $A_1,A_2,\dotsc\in\mathcal{A}$ such that for each $A\in\mathcal{A}$ there exists a $n$ such that $(A\setminus f_A(n))\cap(A_n\setminus f_{A_n}(n))\neq\emptyset$.
\end{enumerate}
\end{Th}

\subsection{Groupability and finite powers}
\begin{Th}[{\!\cite[Theorem 2.1]{sHR}}]
\label{TN33}
If every finite power of a space $X$ is star-Menger, then $X$ satisfies $U_{fin}^*(\mathcal{O},\Omega)$.
\end{Th}

\begin{Th}[{\!\cite[Theorem 2.2]{sHR}}]
\label{TN34}
For a space $X$ the following assertions are equivalent.
\begin{enumerate}[wide=0pt,label={\upshape(\arabic*)},leftmargin=*]
  \item $X$ satisfies $U_{fin}^*(\mathcal{O},\Omega)$.
  \item $X$ satisfies $U_{fin}^*(\mathcal{O},\mathcal{O}^{wgp})$.
\end{enumerate}
\end{Th}

\begin{Cor}
\label{CN3}
If every finite power of a space $X$ is star-Menger, then $X$ satisfies $U_{fin}^*(\mathcal{O},\mathcal{O}^{wgp})$.
\end{Cor}

\begin{Th}[{\!\cite[Theorem 4.3]{sHR}}]
\label{TN60}
For a space $X$ the following assertions are equivalent.
\begin{enumerate}[wide=0pt,label={\upshape(\arabic*)},leftmargin=*]
  \item $X$ is star-Hurewicz.
  \item $X$ satisfies $U_{fin}^*(\mathcal{O},\mathcal{O}^{gp})$.
\end{enumerate}
\end{Th}

\begin{Th}
\label{TN28}
If every finite power of a space $X$ is star-Rothberger, then $X$ satisfies $S_1^*(\Omega,\Omega)$.
\end{Th}
\begin{proof}
Consider a sequence $(\mathcal{U}_n)$ of $\omega$-covers of $X$. Let $\{N_k : k\in\mathbb{N}\}$ be a partition of $\mathbb{N}$ into infinite subsets. Now for each $k\in\mathbb{N}$ and for each $n\in N_k$, $\mathcal{W}_n=\{U^k : U\in\mathcal{U}_n\}$ is an open cover of $X^k$. Let $k\in\mathbb{N}$. Apply the star-Rothberger property of $X^k$ to $(\mathcal{W}_n : n\in N_k)$ to obtain a sequence $(U_n : n\in N_k)$ such that for each $n\in N_k$, $U_n\in\mathcal{U}_n$ and $\{St(U_n^k,\mathcal{W}_n) : n\in N_k\}$ covers $X^k$. It now remains to show that $\{St(U_n,\mathcal{U}_n) : n\in\mathbb{N}\}$ is an $\omega$-cover of $X$. Let $F=\{x_1,x_2,\dotsc,x_p\}$ be a finite subset of $X$. Now $\langle x_1,x_2,\dotsc,x_p\rangle\in X^p$ gives a $n_0\in N_p$ such that $\langle x_1,x_2,\dotsc,x_p\rangle\in St(U_{n_0}^p,\mathcal{W}_{n_0})$. It follows that $\langle x_1,x_2,\dotsc,x_p\rangle\in U^p$ for some $U\in\mathcal{U}_{n_0}$ with $U^p\cap U_{n_0}^p\neq\emptyset$. Clearly $F\subseteq St(U_{n_0},\mathcal{U}_{n_0})$.
\end{proof}

\begin{Cor}
If every finite power of a space $X$ is star-Rothberger, then $X$ satisfies $S_1^*(\Omega,\mathcal{O}^{wgp})$.
\end{Cor}

\begin{Th}
\label{T4}
If every finite power of a space $X$ is star-K-Menger, then $X$ satisfies $SS_{comp}^*(\mathcal{O},\Omega)$.
\end{Th}
\begin{proof}
Let $(\mathcal{U}_n)$ be a sequence of open covers of $X$. Let $\{N_k : k\in\mathbb{N}\}$ be a partition of $\mathbb{N}$ into infinite subsets. For each $k\in\mathbb{N}$ and each $m\in N_k$, define $\mathcal{W}_m=\{U_1\times U_2\times\dotsb\times U_k : U_1,U_2,\ldots,U_k\in\mathcal{U}_m\}$. Now $(\mathcal{W}_m:{m\in N_k})$ is a sequence of open covers of $X^k$. Since $X^k$ is star-K-Menger, there exists a sequence $(K_m:{m\in N_k})$ of compact subsets of $X^k$ such that $\{St(K_m,\mathcal{W}_m) : m\in N_k\}$ is an open cover of $X^k$. For each $1\leq i\leq k$, let $p_i:X^k\to X$ be the $i$th projection mapping. For each $m\in N_k$, $C_m=\cup_{1\leq i\leq k}p_i(K_m)$ is compact as $p_i(K_m)$ is compact for each $1\leq i\leq k$ and each $m\in N_k$. Thus $K_m\subseteq C_m^k$ for each $m\in N_k$. It now follows that the $SS_{comp}^*(\mathcal{O},\Omega)$ property of $X$ is witnessed by the sequence $(C_n)$.
\end{proof}

\begin{Th}
\label{T5}
For a space $X$ the following assertions are equivalent.
\begin{enumerate}[label={\upshape(\arabic*)}, leftmargin=*]
  \item $X$ satisfies $SS_{comp}^*(\mathcal{O},\Omega)$.
  \item $X$ satisfies $SS_{comp}^*(\mathcal{O},\mathcal{O}^{wgp})$.
\end{enumerate}
\end{Th}
\begin{proof}
We only prove $(2)\Rightarrow(1)$. Let $(\mathcal{U}_n)$ be a sequence of open covers of $X$. For each $n$ let $\mathcal{V}_n=\wedge_{i\leq n}\mathcal{U}_i$. For each $n$ $\mathcal{V}_n$ is an open cover of $X$ which refines $\mathcal{U}_i$ for each $i\leq n$. Apply $(2)$ to the sequence $(\mathcal{V}_n)$ to obtain a sequence $(K_n)$ of compact subsets of $X$ such that $\{St(K_n,\mathcal{V}_n) : n\in\mathbb{N}\}$ is a weakly groupable cover of $X$. Subsequently there is a sequence $n_1<n_2<\cdots<n_k<\cdots$ of positive integers such that for every finite set $F\subseteq X$ we have $F\subseteq\cup_{n_k\leq i\leq n_{k+1}}St(K_i,\mathcal{V}_i)$ for some $k\in\mathbb{N}$. Define
\[
C_n=
\begin{cases}
 \cup_{i<n_1}K_i & \mbox{for } n<n_1  \\
  \cup_{n_k\leq i<n_{k+1}}K_i & \mbox{for } n_k\leq n<n_{k+1}.
\end{cases}
\]
An easy verification shows that the sequence $(C_n)$ of compact subsets fulfills the requirement.
\end{proof}

In combination with Theorem~\ref{T4} we obtain the following.
\begin{Cor}
\label{C1}
If every finite power of a space $X$ is star-K-Menger, then $X$ satisfies $SS_{comp}^*(\mathcal{O},\mathcal{O}^{wgp})$.
\end{Cor}

The proof of the following result is similar to Theorem~\ref{T5} with necessary modifications.
\begin{Th}
\label{TN32}
For a space $X$ the following assertions are equivalent.
\begin{enumerate}[wide=0pt,label={\upshape(\arabic*)},leftmargin=*]
  \item $X$ is star-K-Hurewicz.
  \item $X$ satisfies $SS_{comp}^*(\mathcal{O},\mathcal{O}^{gp})$.
\end{enumerate}
\end{Th}

\begin{Th}[{\!\cite[Theorem 3.1]{sHR}}]
\label{TN35}
If every finite power of a space $X$ is strongly star-Menger, then $X$ satisfies $SS_{fin}^*(\mathcal{O},\Omega)$.
\end{Th}

\begin{Th}[{\!\cite[Theorem 3.2]{sHR}}]
\label{TN36}
For a space $X$ the following assertions are equivalent.
\begin{enumerate}[wide=0pt,label={\upshape(\arabic*)},leftmargin=*]
  \item $X$ satisfies $SS_{fin}^*(\mathcal{O},\Omega)$.
  \item $X$ satisfies $SS_{fin}^*(\mathcal{O},\mathcal{O}^{wgp})$.
\end{enumerate}
\end{Th}

\begin{Cor}
\label{CN4}
If every finite power of a space $X$ is strongly star-Menger, then $X$ satisfies $SS_{fin}^*(\mathcal{O},\mathcal{O}^{wgp})$.
\end{Cor}

\begin{Th}[{\!\cite[Theorem 5.2]{sHR}}]
\label{TN37}
For a space $X$ the following assertions are equivalent.
\begin{enumerate}[wide=0pt,label={\upshape(\arabic*)},leftmargin=*]
  \item $X$ is strongly star-Hurewicz.
  \item $X$ satisfies $SS_{fin}^*(\mathcal{O},\mathcal{O}^{gp})$.
\end{enumerate}
\end{Th}

In line of Theorem~\ref{TN28}, we obtain the following.
\begin{Th}
\label{TN29}
If every finite power of a space $X$ is strongly star-Rothberger, then $X$ satisfies $SS_1^*(\Omega,\Omega)$.
\end{Th}

\begin{Cor}
If every finite power of a space $X$ is strongly star-Rothberger, then $X$ satisfies $SS_1^*(\Omega,\mathcal{O}^{wgp})$.
\end{Cor}

\begin{table}[h!]
\centering
  \begin{tabular}{|l|l|l|c|}
  \hline
  $(\forall n)$ $X^n\models$ &\multicolumn{2}{c|}{$X\models$} & Source \\
  \hline
  star-Menger &  $U_{fin}^*(\mathcal{O},\Omega)$ & $=U_{fin}^*(\mathcal{O},\mathcal{O}^{wgp})$ & \cite{sHR}\\
   \hline
  star-K-Menger & $SS_{comp}^*(\mathcal{O},\Omega)$ &  $=SS_{comp}^*(\mathcal{O},\mathcal{O}^{wgp})$ & \\
   \hline
  strongly star-Menger & $SS_{fin}^*(\mathcal{O},\Omega)$ & $=SS_{fin}^*(\mathcal{O},\mathcal{O}^{wgp})$ & \cite{sHR}\\
   \hline
  star-Rothberger & $S_1^*(\Omega,\Omega)$ & $\Rightarrow S_1^*(\Omega,\mathcal{O}^{wgp})$ & \\
   \hline
  strongly star-Rothberger & $SS_1^*(\Omega,\Omega)$ & $\Rightarrow SS_1^*(\Omega,\mathcal{O}^{wgp})$ &\\
   \hline
\end{tabular}
\vskip0.1cm
\caption{Property in finite powers}
\label{tab1}
\end{table}

\begin{table}[h!]
\begin{tabular}{|l|l|c|}
\hline
  Property & Equivalent to & Source\\
  \hline
  star-Hurewicz &  $U_{fin}^*(\mathcal{O},\mathcal{O}^{gp})$ &\cite{sHR}\\
  \hline
  star-K-Hurewicz & $SS_{comp}^*(\mathcal{O},\mathcal{O}^{gp})$&\\
  \hline
  strongly star-Hurewicz & $SS_{fin}^*(\mathcal{O},\mathcal{O}^{gp})$ &\cite{sHR}\\
  \hline
\end{tabular}
\vskip0.1cm
\caption{Classification using groupable covers}
\label{tab2}
\end{table}

\subsection{Mappings, products and the Alexandroff duplicate}
Each of the star variants described in Figure~\ref{dig1} is preserved under clopen subsets, countable unions and continuous mappings (see \cite{rsM-II,RSH,sHR,SR,sKM,sKH,rssM,RSSH,LjSM,survey}). In particular we mention the following preservation under open perfect mappings.

\begin{Th}[{\!\cite[Theorem 2.10]{rsM-II}}]
\label{TN46}
If $f$ is an open perfect mapping from $X$ onto a star-Menger space $Y,$ then $X$ is also star-Menger.
\end{Th}

\begin{Th}[{\!\cite[Theorem 2.8]{RSH}}]
\label{TN47}
If $f$ is an open perfect mapping from $X$ onto a star-Hurewicz space $Y,$ then $X$ is also star-Hurewicz.
\end{Th}

\begin{Th}[{\!\cite[Theorem 3.4]{sKM}}]
\label{TN48}
If $f$ is an open perfect mapping from $X$ onto a star-K-Menger space $Y,$ then $X$ is also star-K-Menger.
\end{Th}

\begin{Th}[{\!\cite[Theorem 2.15]{sKH}}]
\label{TN49}
If $f$ is an open perfect mapping from $X$ onto a star-K-Hurewicz space $Y,$ then $X$ is also star-K-Hurewicz.
\end{Th}

The strongly star-Menger and strongly star-Hurewicz properties are not inverse invariant of open perfect mappings (see \cite[Remark 2.15]{rssM} and \cite[Remark 2.9]{RSSH}).
In view of these observations, it follows that the product of a star-Menger (respectively, star-Hurewicz, star-K-Menger, star-K-Hurewicz) space with a $\sigma$-compact space is again star-Menger (respectively, star-Hurewicz, star-K-Menger, star-K-Hurewicz).

Let $P$ be a property of a space. A space $X$ is called $P$-preserving if for every space $Y$ with property $P$, $X\times Y$ has the property $P$.
\begin{table}[h!]
\begin{tabular}{|l|c|c|}
\hline
  Property & Property-preserving class & Source\\
  \hline
  star-Menger & \multirow[c]{4}{*}{$\sigma$-compact} & \cite{rsM-II}\\
  \cline{1-1}\cline{3-3}
  star-K-Menger & & \cite{sKM}\\
  \cline{1-1}\cline{3-3}
  star-Hurewicz & & \cite{RSH}\\
  \cline{1-1}\cline{3-3}
  star-K-Hurewicz & & \cite{sKH}\\
  \hline
\end{tabular}
\vskip0.1cm
\caption{Property-preserving class: $\sigma$-compact}
\label{tab3}
\end{table}

By \cite[Remark 2.15]{rssM} (respectively, \cite[Remark 2.9]{RSSH}), the product of a strongly star-Menger (respectively, strongly star-Hurewicz) space and a compact space need not be strongly star-Menger (respectively, strongly star-Hurewicz). We now observe that the star-Rothberger and strongly star-Rothberger properties are not inverse invariant of open perfect mappings. Indeed, if possible suppose that the star-Rothberger and strongly star-Rothberger properties are inverse invariant of open perfect mappings. It follows that product of a star-Rothberger (respectively, strongly star-Rothberger) space with a $\sigma$-compact space is star-Rothberger (respectively, strongly star-Rothberger). Now assume that $X$ is any star-Rothberger (respectively, strongly star-Rothberger) space and $Y=\mathbb{R}$ is the set of reals. Thus $X\times Y$ is star-Rothberger (respectively, strongly star-Rothberger). Since the star-Rothberger and strongly star-Rothberger properties are preserved under continuous mappings, $Y$ is star-Rothberger (respectively, strongly star-Rothberger). Which is absurd because the set of reals is not star-Rothberger (hence not strongly star-Rothberger). Consequently the star-Rothberger and strongly star-Rothberger properties are not inverse invariant of open perfect mappings. This also shows that product of a star-Rothberger (respectively, strongly star-Rothberger) space with a compact space need not be star-Rothberger (respectively, strongly star-Rothberger), which contradicts the following result of Ko\v{c}inac \cite{LjSM} (given without proof).

\begin{Th}[{\!\cite[Theorem 2.13]{LjSM}}]
\label{TN38}
If $X$ is a star-Rothberger space and $Y$ is a compact space, then $X\times Y$ is a star-Rothberger space.
\end{Th}

Under open closed finite-to-one continuous mappings the star-Menger, star-Hurewicz, star-K-Menger and star-K-Hurewicz properties are inverse invariant. Also we have the following.
\begin{Th}[{\!\cite[Theorem 2.13]{rssM}}]
\label{TN50}
If $f$ is an open closed finite-to-one continuous mapping from $X$ onto a strongly star-Menger space $Y,$ then $X$ is also strongly star-Menger.
\end{Th}

\begin{Th}[{\!\cite[Theorem 2.6]{RSSH}}]
\label{TN51}
If $f$ is an open closed finite-to-one continuous mapping from $X$ onto a strongly star-Hurewicz space $Y,$ then $X$ is also strongly star-Hurewicz.
\end{Th}

\begin{Prob}\rm
\label{Q7}
Are the star-Rothberger and strongly star-Rothberger properties inverse invariant under open closed finite-to-one continuous mappings?
\end{Prob}

We now move on to Alexandroff duplicate $A(X)$. Let $P$ be any of the star variants of Figure~\ref{dig1}. If $X$ satisfies $P$, then $A(X)$ need not satisfy $P$. See \cite{rsM,rsM-II,RSKM,sKH,rssM} for such counterexamples for the star versions of the Menger and Hurewicz properties. For the star variants of the Rothberger property we obtain the following.

\begin{Ex}\rm
\label{EN1}
There exists a Tychonoff star-Rothberger (respectively, strongly star-Rothberger) space $X$ such that $A(X)$ is not star-Rothberger (respectively, strongly star-Rothberger).
\end{Ex}
\begin{proof}
Assume that $\omega_1<\cov(\mathcal{M})$. Let $X=\Psi(\mathcal{A})$ with $|\mathcal{A}|=\omega_1$. By Theorem~\ref{TN56}, $X$ is strongly star-Rothberger (hence star-Rothberger). But $A(X)$ is not star-Rothberger (hence not strongly star-Rothberger). Indeed, $\mathcal{A}\times\{1\}$ is a clopen discrete subset of $A(X)$ with cardinality $\omega_1$ and the star-Rothberger property is preserved under clopen subsets.
\end{proof}

It is natural to consider the following questions.

\begin{Prob}[{\!\cite[Remark 2.2]{rsM-II}}]\rm
\label{Q3}
Is there a space $X$ such that $A(X)$ is star-Menger, but $X$ is not star-Menger?
\end{Prob}

\begin{Prob}\rm
\label{Q2}
Is there a space $X$ such that $A(X)$ is star-Hurewicz (respectively, star-Rothberger), but $X$ is not star-Hurewicz (respectively, star-Rothberger)?
\end{Prob}

\begin{Prob}\rm
\label{Q5}
Is there a space $X$ such that $A(X)$ is star-K-Menger (respectively, star-K-Hurewicz), but $X$ is not star-K-Menger (respectively, star-K-Hurewicz)?
\end{Prob}

\begin{Prob}[{\!\cite[Remark 2.10]{rssM}}]\rm
\label{Q4}
Is there a space $X$ such that $A(X)$ is strongly star-Menger, but $X$ is not strongly star-Menger?
\end{Prob}

\begin{Prob}\rm
\label{Q6}
Is there a space $X$ such that $A(X)$ is strongly star-Hurewicz (respectively, strongly star-Rothberger), but $X$ is not strongly star-Hurewicz (respectively, strongly star-Rothberger)?
\end{Prob}

The following results (Theorem~\ref{T3} to Theorem~\ref{TN45}) show that the answers to the above problems are not affirmative.
\begin{Th}
\label{T3}
If $A(X)$ is star-Menger (respectively, star-Hurewicz, star-Rothberger), then $X$ is also star-Menger (respectively, star-Hurewicz, star-Rothberger).
\end{Th}
\begin{proof}
We only present proof for the star-Rothberger case. Consider a sequence $(\mathcal{U}_n)$ of open covers of $X$. For each $n$ $\mathcal{W}_n=\{U\times\{0,1\} : U\in\mathcal{U}_n\}$ is an open cover of $A(X)$. Apply the star-Rothberger property of $A(X)$ to $(\mathcal{W}_n)$ to obtain a sequence $(V_n)$ such that for each $n$ $V_n\in\mathcal{W}_n$ and $\{St(V_n,\mathcal{W}_n) : n\in\mathbb{N}\}$ covers $A(X)$. For each $n$ choose a $U_n\in\mathcal{U}_n$ such that $V_n=U_n\times\{0,1\}$. The sequence $(U_n)$ witnesses for $(\mathcal{U}_n)$ that $X$ is star-Rothberger.
\end{proof}
Analogously we obtain next two results.
\begin{Th}
\label{TN44}
If $A(X)$ is star-K-Menger (respectively, star-K-Hurewicz), then $X$ is also star-K-Menger (respectively, star-K-Hurewicz).
\end{Th}

\begin{Th}
\label{TN45}
If $A(X)$ is strongly star-Menger (respectively, strongly star-Hurewicz, strongly star-Rothberger), then $X$ is also strongly star-Menger (respectively, strongly star-Hurewicz, strongly star-Rothberger).
\end{Th}

\subsection{Alster covers}
We recall the notions of the following covers from \cite{Alster}.
\begin{enumerate}[leftmargin=*]
\item[$\mathcal{G}$:] The family of all covers $\mathcal{U}$ of the space $X$ for which each element of $\mathcal{U}$ is a $G_\delta$ set.

\item[$\mathcal{G}_K$:] The family of all Alster covers of $X$. A cover $\mathcal{U}\in\mathcal{G}$ is said to be an Alster cover if $X$ is not in $\mathcal{U}$ and for each compact set $C\subseteq X$ there is a $U\in\mathcal{U}$ such that $C\subseteq U$.

\item[$\mathcal{G}_\Gamma$:] The family of all covers $\mathcal{U}\in\mathcal{G}$ which are infinite and each infinite subset of $\mathcal{U}$ is a cover of $X$.
\end{enumerate}

For a family $\mathcal{U}$ of open covers of $X$ we use the symbol $\mathcal{U}_\cup$ to denote the collection of all such members of $\mathcal{U}$ which are closed under finite unions.
We start with the following lemmas which will be used subsequently.
\begin{Lemma}
\label{LN9}
For a space $X$ the following assertions are equivalent.
\begin{enumerate}[wide=0pt,label={\upshape(\arabic*)},leftmargin=*]
  \item $X$ satisfies $S_{fin}^*(\mathcal{O}_\cup,\mathcal{O})$.
  \item $X$ satisfies $U_{fin}^*(\mathcal{O}_\cup,\Omega)$.
  \item $X$ satisfies $S_1^*(\mathcal{O}_\cup,\mathcal{O})$.
  \item $X$ satisfies $S_1^*(\mathcal{O}_\cup,\Omega)$.
\end{enumerate}
\end{Lemma}
\begin{proof}
$(1)\Rightarrow (2)$. Let $(\mathcal{U}_n)$ be a sequence of members of $\mathcal{O}_\cup$. Since $X$ satisfies $S_{fin}^*(\mathcal{O}_\cup,\mathcal{O})$, there is a sequence $(\mathcal{V}_n)$ such that for each $n$ $\mathcal{V}_n$ is a finite subset of $\mathcal{U}_n$ and $\{St(\cup\mathcal{V}_n,\mathcal{U}_n) : n\in\mathbb{N}\}$ covers $X$. Let $F$ be a finite subset of $X$. For each $x\in F$ choose a $n_x\in\mathbb{N}$ such that $x\in St(\cup\mathcal{V}_{n_x},\mathcal{U}_{n_x})$. Let $a\in F$. We can find a $U_a\in\mathcal{U}_{n_a}$ with $a\in U_a$ and $U_a\cap(\cup\mathcal{V}_{n_a})\neq\emptyset$. For each $x\in F\setminus\{a\}$ choose a $U_x\in\mathcal{U}_{n_a}$ containing $x$. Clearly $U=U_a\cup(\cup_{x\in F\setminus\{a\}}U_x)\in\mathcal{U}_{n_a}$ and $U\cap(\cup\mathcal{V}_{n_a})\neq\emptyset$, and accordingly $F\subseteq St(\cup\mathcal{V}_{n_a},\mathcal{U}_{n_a})$. Thus $\{St(\cup\mathcal{V}_n,\mathcal{U}_n) : n\in\mathbb{N}\}$ is an $\omega$-cover and $X$ satisfies $U_{fin}^*(\mathcal{O}_\cup,\Omega)$.

$(1)\Rightarrow (3)$. Let $(\mathcal{U}_n)$ be a sequence of members of $\mathcal{O}_\cup$. Since $X$ satisfies $S_{fin}^*(\mathcal{O}_\cup,\mathcal{O})$, there exists a sequence $(\mathcal{V}_n)$ such that for each $n$ $\mathcal{V}_n$ is a finite subset of $\mathcal{U}_n$ and $\{St(\cup\mathcal{V}_n,\mathcal{U}_n) : n\in\mathbb{N}\}$ covers $X$. For each $n$ let $V_n=\cup\mathcal{V}_n$. Since for each $n$ $\cup\mathcal{V}_n\in\mathcal{U}_n$, the sequence $(V_n)$ witnesses for $(\mathcal{U}_n)$ that $X$ satisfies $S_1^*(\mathcal{O}_\cup,\mathcal{O})$.

$(3)\Rightarrow (4)$. Let $(\mathcal{U}_n)$ be a sequence of members of $\mathcal{O}_\cup$. Since $X$ satisfies $S_1^*(\mathcal{O}_\cup,\mathcal{O})$, there exists a sequence $(U_n)$ such that for each $n$ $U_n\in\mathcal{U}_n$ and $\{St(U_n,\mathcal{U}_n) : n\in\mathbb{N}\}$ covers $X$. Let $F$ be a finite subset of $X$. For each $x\in F$ choose a $n_x\in\mathbb{N}$ such that $x\in St(U_{n_x},\mathcal{U}_{n_x})$. Let $a\in F$. We get a $U_a\in\mathcal{U}_{n_a}$ with $a\in U_a$ and $U_a\cap U_{n_a}\neq\emptyset$. For each $x\in F\setminus\{a\}$ we pick a $U_x\in\mathcal{U}_{n_a}$ containing $x$. Clearly $U=U_a\cup(\cup_{x\in F\setminus\{a\}}U_x)\in\mathcal{U}_{n_a}$ with $U\cap U_{n_a}\neq\emptyset$ and thus $F\subseteq St(U_{n_a},\mathcal{U}_{n_a})$. This implies that $\{St(U_n,\mathcal{U}_n) : n\in\mathbb{N}\}$ is an $\omega$-cover of $X$. Hence $X$ satisfies $S_1^*(\mathcal{O}_\cup,\Omega)$.
\end{proof}

\begin{Lemma}
\label{LN1}
For a space $X$ the following assertions are equivalent.
\begin{enumerate}[wide=0pt,label={\upshape(\arabic*)},leftmargin=*]
  \item $X$ satisfies $SS_{comp}^*(\mathcal{O}_\cup,\mathcal{O})$.
  \item $X$ satisfies $SS_{comp}^*(\mathcal{O}_\cup,\Omega)$.
\end{enumerate}
\end{Lemma}

\begin{Lemma}
\label{LN10}
For a space $X$ the following assertions are equivalent.
\begin{enumerate}[wide=0pt,label={\upshape(\arabic*)},leftmargin=*]
  \item $X$ satisfies $SS_{fin}^*(\mathcal{O}_\cup,\mathcal{O})$.
  \item $X$ satisfies $SS_{fin}^*(\mathcal{O}_\cup,\Omega)$.
  \item $X$ satisfies $SS_1^*(\mathcal{O}_\cup,\mathcal{O})$.
  \item $X$ satisfies $SS_1^*(\mathcal{O}_\cup,\Omega)$.
\end{enumerate}
\end{Lemma}
\begin{proof}
$(1)\Rightarrow (2)$. Let $(\mathcal{U}_n)$ be a sequence of members of $\mathcal{O}_\cup$. Since $X$ satisfies $SS_{fin}^*(\mathcal{O}_\cup,\mathcal{O})$, there exists a sequence $(F_n)$ of finite subsets of $X$ such that $\{St(F_n,\mathcal{U}_n) : n\in\mathbb{N}\}$ covers $X$. Let $F$ be a finite subset of $X$. For each $x\in F$ choose a $n_x\in\mathbb{N}$ such that $x\in St(F_{n_x},\mathcal{U}_{n_x})$. Let $a\in F$. We get a $U_a\in\mathcal{U}_{n_a}$ with $a\in U_a$ and $U_a\cap F_{n_a}\neq\emptyset$. For each $x\in F\setminus\{a\}$ we pick a $U_x\in\mathcal{U}_{n_a}$ containing $x$. Clearly $U=U_a\cup(\cup_{x\in F\setminus\{a\}}U_x)\in\mathcal{U}_{n_a}$ with $U\cap F_{n_a}\neq\emptyset$. Thus $F\subseteq St(F_{n_a},\mathcal{U}_{n_a})$ and $\{St(F_n,\mathcal{U}_n) : n\in\mathbb{N}\}$ is an $\omega$-cover of $X$. Consequently $X$ satisfies $SS_{fin}^*(\mathcal{O}_\cup,\Omega)$.

$(1)\Rightarrow (3)$. Let $(\mathcal{U}_n)$ be a sequence of members of $\mathcal{O}_\cup$. Apply the property $SS_{fin}^*(\mathcal{O}_\cup,\mathcal{O})$ to $(\mathcal{U}_n)$ to obtain a sequence $(F_n)$ of finite subsets of $X$ such that $\{St(F_n,\mathcal{U}_n) : n\in\mathbb{N}\}$ covers $X$. For each $n$ choose a $x_n\in F_n$. We claim that the sequence $(x_n)$ guarantees for $(\mathcal{U}_n)$ that $X$ satisfies $SS_1^*(\mathcal{O}_\cup,\mathcal{O})$. Let $x\in X$. Pick a $n_0\in\mathbb{N}$ such that $x\in St(F_{n_0},\mathcal{U}_{n_0})$. Then we get a $V\in\mathcal{U}_{n_0}$ such that $x\in V$ and $V\cap F_{n_0}\neq\emptyset$. Also choose a $W\in\mathcal{U}_{n_0}$ such that $x_{n_0}\in W$. Thus we obtain a $U=V\cup W\in\mathcal{U}_{n_0}$ with $x, x_{n_0}\in U$. Subsequently $x\in St(x_{n_0},\mathcal{U}_{n_0})$ and $\{St(x_n,\mathcal{U}_n) : n\in\mathbb{N}\}$ covers $X$. Hence $X$ satisfies $SS_1^*(\mathcal{O}_\cup,\mathcal{O})$.
%

$(3)\Rightarrow (4)$. Let $(\mathcal{U}_n)$ be a sequence of members of $\mathcal{O}_\cup$. Since $X$ satisfies $SS_1^*(\mathcal{O}_\cup,\mathcal{O})$, there exists a sequence $(x_n)$ of elements of $X$ such that $\{St(x_n,\mathcal{U}_n) : n\in\mathbb{N}\}$ covers $X$. Let $F$ be a finite subset of $X$. For each $x\in F$ choose a $n_x\in\mathbb{N}$ such that $x\in St(x_{n_x},\mathcal{U}_{n_x})$. Let $a\in F$. We get a $U_a\in\mathcal{U}_{n_a}$ with $x_{n_a}, a\in U_a$. For each $x\in F\setminus\{a\}$ we pick a $U_x\in\mathcal{U}_{n_a}$ containing $x$. Clearly $U=U_a\cup(\cup_{x\in F\setminus\{a\}}U_x)\in\mathcal{U}_{n_a}$ with $x_{n_a}\in U$ and thus $F\subseteq St(x_{n_a},\mathcal{U}_{n_a})$. It follows that $\{St(x_n,\mathcal{U}_n) : n\in\mathbb{N}\}$ is an $\omega$-cover of $X$. Hence $X$ satisfies $SS_1^*(\mathcal{O}_\cup,\Omega)$.
\end{proof}

By Theorem~\ref{TN34} and Lemma~\ref{LN9}, we obtain the following.
\begin{Th}
\label{TN39}
For a space $X$ the following assertions are equivalent.
\begin{enumerate}[wide=0pt,label={\upshape(\arabic*)},leftmargin=*]
  \item $X$ satisfies $S_{fin}^*(\mathcal{O}_\cup,\mathcal{O})$.
  \item $X$ satisfies $U_{fin}^*(\mathcal{O}_\cup,\Omega)$.
  \item $X$ satisfies $U_{fin}^*(\mathcal{O}_\cup,\mathcal{O}^{wgp})$.
  \item $X$ satisfies $S_1^*(\mathcal{O}_\cup,\mathcal{O})$.
  \item $X$ satisfies $S_1^*(\mathcal{O}_\cup,\Omega)$.
  \item $X$ satisfies $S_1^*(\mathcal{O}_\cup,\mathcal{O}^{wgp})$.
\end{enumerate}
\end{Th}

Also by Theorem~\ref{T5} and Lemma~\ref{LN1}, we obtain the following.
\begin{Th}
\label{TN52}
For a space $X$ the following assertions are equivalent.
\begin{enumerate}[wide=0pt,label={\upshape(\arabic*)},leftmargin=*]
  \item $X$ satisfies $SS_{comp}^*(\mathcal{O}_\cup,\mathcal{O})$.
  \item $X$ satisfies $SS_{comp}^*(\mathcal{O}_\cup,\Omega)$.
  \item $X$ satisfies $SS_{comp}^*(\mathcal{O}_\cup,\mathcal{O}^{wgp})$.
\end{enumerate}
\end{Th}

Next result follows similarly from Theorem~\ref{TN36} and Lemma~\ref{LN10}.
\begin{Th}
\label{TN40}
For a space $X$ the following assertions are equivalent.
\begin{enumerate}[wide=0pt,label={\upshape(\arabic*)},leftmargin=*]
  \item $X$ satisfies $SS_{fin}^*(\mathcal{O}_\cup,\mathcal{O})$.
  \item $X$ satisfies $SS_{fin}^*(\mathcal{O}_\cup,\Omega)$.
  \item $X$ satisfies $SS_{fin}^*(\mathcal{O}_\cup,\mathcal{O}^{wgp})$.
  \item $X$ satisfies $SS_1^*(\mathcal{O}_\cup,\mathcal{O})$.
  \item $X$ satisfies $SS_1^*(\mathcal{O}_\cup,\Omega)$.
  \item $X$ satisfies $SS_1^*(\mathcal{O}_\cup,\mathcal{O}^{wgp})$.
\end{enumerate}
\end{Th}

\begin{Th}
\label{TN2}
If a space $X$ has the property $S_1^*(\mathcal{G}_K,\mathcal{G})$, then $X$ satisfies $U_{fin}^*(\mathcal{O}_\cup,\Omega)$.
\end{Th}
\begin{proof}
Let $(\mathcal{U}_n)$ be a sequence of members of $\mathcal{O}_\cup$. Observe that $\mathcal{U}=\{\cap_{n\in\mathbb{N}}U_n : U_n\in\mathcal{U}_n\}\in\mathcal{G}$. Also if $C$ is a compact subset of $X$, then for each $n$ we can choose a finite set $\mathcal{V}_n\subseteq\mathcal{U}_n$ such that $C\subseteq\cup\mathcal{V}_n$. Now $\mathcal{U}\in\mathcal{G}_K$ as $\cup\mathcal{V}_n\in\mathcal{U}_n$ for each $n$. Since $X$ satisfies $S_1^*(\mathcal{G}_K,\mathcal{G})$, there is a sequence $(V_n)$ of members of $\mathcal{U}$ such that $\{St(V_n,\mathcal{U}) : n\in\mathbb{N}\}$ covers $X$. For each $n$ let $V_n=\cap_{k\in\mathbb{N}}U_k^{(n)}$, where for each $k$ $U_k^{(n)}\in\mathcal{U}_k$. Define a sequence $(W_n)$ as $W_n=U_n^{(n)}\in\mathcal{U}_n$ for each $n$. We now show that the sequence $(W_n)$ witnesses for $(\mathcal{U}_n)$ that $X$ satisfies $S_1^*(\mathcal{O}_\cup,\mathcal{O})$. Let $x\in X$ and choose a $n_0\in\mathbb{N}$ such that $x\in St(V_{n_0},\mathcal{U})$, i.e. there exists a $U\in\mathcal{U}$ containing $x$ such that $U\cap V_{n_0}\neq\emptyset$. Since $V_{n_0}\subseteq W_{n_0}$ and $U\subseteq V$ for some $V\in\mathcal{U}_{n_0}$, $x\in St(W_{n_0},\mathcal{U}_{n_0})$. Thus $X$ satisfies $S_1^*(\mathcal{O}_\cup,\mathcal{O})$. By Lemma~\ref{LN9}, $X$ satisfies $U_{fin}^*(\mathcal{O}_\cup,\Omega)$.
\end{proof}

\begin{Cor}
\label{CN5}
If a space $X$ has the property $S_1^*(\mathcal{G}_K,\mathcal{G})$, then $X$ satisfies any one of the following properties.
\begin{enumerate}[wide=0pt,label={\upshape(\arabic*)},leftmargin=*]
  \item $S_{fin}^*(\mathcal{O}_\cup,\mathcal{O})$.
  \item $U_{fin}^*(\mathcal{O}_\cup,\mathcal{O}^{wgp})$.
  \item $S_1^*(\mathcal{O}_\cup,\mathcal{O})$.
  \item $S_1^*(\mathcal{O}_\cup,\Omega)$.
  \item $S_1^*(\mathcal{O}_\cup,\mathcal{O}^{wgp})$.
\end{enumerate}
\end{Cor}

\begin{Th}
\label{TN3}
If a space $X$ has the property $SS_1^*(\mathcal{G}_K,\mathcal{G})$, then $X$ satisfies $SS_{fin}^*(\mathcal{O}_\cup,\Omega)$.
\end{Th}

\begin{Cor}
\label{CN6}
If a space $X$ has the property $SS_1^*(\mathcal{G}_K,\mathcal{G})$, then $X$ satisfies any one of the following properties.
\begin{enumerate}[wide=0pt,label={\upshape(\arabic*)},leftmargin=*]
  \item $SS_{fin}^*(\mathcal{O}_\cup,\mathcal{O})$.
  \item $SS_{fin}^*(\mathcal{O}_\cup,\mathcal{O}^{wgp})$.
  \item $SS_1^*(\mathcal{O}_\cup,\mathcal{O})$.
  \item $SS_1^*(\mathcal{O}_\cup,\Omega)$.
  \item $SS_1^*(\mathcal{O}_\cup,\mathcal{O}^{wgp})$.
\end{enumerate}
\end{Cor}

\begin{Th}
\label{TN4}
If $X$ satisfies $S_1^*(\mathcal{G}_K,\mathcal{G}_\Gamma)$ and $Y$ satisfies $S_{fin}^*(\mathcal{O}_\cup,\mathcal{O})$, then $X\times Y$ satisfies $S_{fin}^*(\mathcal{O}_\cup,\mathcal{O})$.
\end{Th}
\begin{proof}
Let $(\mathcal{U}_n)$ be a sequence of members of $\mathcal{O}_\cup$ for $X\times Y$. Now choose two sequences respectively $(\mathcal{A}_n)$ and $(\mathcal{B}_n)$ of open covers of $X$ and $Y$ such that if $A\in\mathcal{A}_n$ (respectively, $B\in\mathcal{B}_n$), there is a $B\in\mathcal{B}_n$ (respectively, a $A\in\mathcal{A}_n$) and a $U\in\mathcal{U}_n$ such that $A\times B\subseteq U$, and if $U\in\mathcal{U}_n$, then there is a $A\in\mathcal{A}_n$ and a $B\in\mathcal{B}_n$ such that $A\times B\subseteq U$.

Let $K$ be a compact subset of $X$. For each $y\in Y$ and each $n$ there is a $U\in\mathcal{U}_n$ such that $K\times\{y\}\subseteq U$. Next choose an open set $V$ in $X$ containing $K$ such that $K\times\{y\}\subseteq V\times\{y\}\subseteq U$. For each $n$ there is a finite subset $\{A_i^{(n)} : 1\leq i\leq k_n\}$ of $\mathcal{A}_n$ such that $K\subseteq\cup_{1\leq i\leq k_n}A_i^{(n)}=V_n$ (say). Clearly $K\subseteq V_n\cap V$ for all $n$. Choose $\Phi(K)=\cap_{n\in\mathbb{N}}(V_n\cap V)$. Thus for each compact subset $K$ of $X$ we obtain a $G_\delta$ subset $\Phi(K)$ of $X$ such that $K\subseteq\Phi(K)$. Also observe that for each compact subset $K$ of $X$,  each $y\in Y$ and each $n$ there is a $U\in\mathcal{U}_n$ such that $\Phi(K)\times\{y\}\subseteq U$.

For each $n$ $\mathcal{W}_n=\{\Phi(K) : K\text{ is a compact subset of }X\}$ is a member of $\mathcal{G}_K$. Since $X$ satisfies $S_1^*(\mathcal{G}_K,\mathcal{G}_\Gamma)$, there exists a sequence $(K_n)$ of compact subsets of $X$ such that $\{St(\Phi(K_n),\mathcal{W}_n) : n\in\mathbb{N}\}\in\mathcal{G}_\Gamma$ for $X$. Observe that for each $n$ $\mathcal{O}_n=\{O\subseteq Y : O\text{ is open in }Y\text{ and }\Phi(K_n)\times O\subseteq U\text{ for some }U\in\mathcal{U}_n\}\in\mathcal{O}_\cup$ for $Y$. Since $Y$ satisfies $S_{fin}^*(\mathcal{O}_\cup,\mathcal{O})$, by Lemma~\ref{LN9}, there is a sequence $(\mathcal{H}_n)$ such that for each $n$ $\mathcal{H}_n$ is a finite subset of $\mathcal{O}_n$ and $\{St(\cup\mathcal{H}_n,\mathcal{O}_n) : n\in\mathbb{N}\}$ is an $\omega$-cover in $Y$. For each $n$ and each $O\in\mathcal{H}_n$ we choose a $U(O)\in\mathcal{U}_n$ such that $\Phi(K_n)\times O\subseteq U(O)$. Define $\mathcal{V}_n=\{U(O) : O\in\mathcal{H}_n\}$.

The proof will be complete if we show that the sequence $(\mathcal{V}_n)$ witnesses for $(\mathcal{U}_n)$ that $X\times Y$ satisfies $S_{fin}^*(\mathcal{O}_\cup,\mathcal{O})$. Let $\langle x,y\rangle\in X\times Y$. Now $x\in St(\Phi(K_n),\mathcal{W}_n)$ for all but finitely many $n$ as $\{St(\Phi(K_n),\mathcal{W}_n) : n\in\mathbb{N}\}\in\mathcal{G}_\Gamma$. Also since $\{St(\cup\mathcal{H}_n,\mathcal{O}_n) : n\in\mathbb{N}\}$ is an $\omega$-cover, $y\in St(\cup\mathcal{H}_n,\mathcal{O}_n)$ for infinitely many $n$. Thus there is a $n_0\in\mathbb{N}$ such that $x\in St(\Phi(K_{n_0}),\mathcal{W}_{n_0})$ and $y\in St(\cup\mathcal{H}_{n_0},\mathcal{O}_{n_0})$. Consequently there are a compact set $K\subseteq X$ and a set $O\in\mathcal{O}_{n_0}$ containing $y$ such that $x\in\Phi(K)$, $\Phi(K)\cap\Phi(K_{n_0})\neq\emptyset$ and $O\cap H\neq\emptyset$ for some $H\in\mathcal{H}_{n_0}$. Choose a $U(H)\in\mathcal{V}_{n_0}$ such that $\Phi(K_{n_0})\times H\subseteq U(H)$. Also there is a set $U_1\in\mathcal{U}_{n_0}$ such that $\Phi(K)\times\{y\}\subseteq U_1$. We thus obtain a set $U=U(H)\cup U_1\in\mathcal{U}_{n_0}$ such that $\langle x,y\rangle\in U$ and $U\cap(\cup\mathcal{V}_{n_0})\neq\emptyset$. This shows that $\langle x,y\rangle\in St(\cup\mathcal{V}_{n_0},\mathcal{U}_{n_0})$ and hence $\{St(\cup\mathcal{V}_n,\mathcal{U}_n) : n\in\mathbb{N}\}
$ covers $X\times Y$.
\end{proof}

\begin{Cor}
\label{CN7}
If $X$ satisfies $S_1^*(\mathcal{G}_K,\mathcal{G}_\Gamma)$ and $Y$ satisfies $S_1^*(\mathcal{O}_\cup,\mathcal{O})$, then $X\times Y$ satisfies $S_1^*(\mathcal{O}_\cup,\mathcal{O})$.
\end{Cor}

\begin{Th}
\label{TN24}
If $X$ satisfies $S_1^*(\mathcal{G}_K,\mathcal{G}_\Gamma)$ and $Y$ satisfies $U_{fin}^*(\mathcal{O}_\cup,\Gamma)$, then $X\times Y$ satisfies $U_{fin}^*(\mathcal{O}_\cup,\Gamma)$.
\end{Th}

\begin{Th}
\label{TN1}
If $X$ satisfies $SS_1^*(\mathcal{G}_K,\mathcal{G}_\Gamma)$ and $Y$ satisfies $SS_{comp}^*(\mathcal{O}_\cup,\mathcal{O})$, then $X\times Y$ satisfies $SS_{comp}^*(\mathcal{O}_\cup,\mathcal{O})$.
\end{Th}
\begin{proof}
Let $(\mathcal{U}_n)$ be a sequence of members of $\mathcal{O}_\cup$ for $X\times Y$. Proceeding similarly as in the proof of Theorem~\ref{TN4}, for each compact set $K\subseteq X$ we obtain a $G_\delta$ set $\Phi(K)\subseteq X$ such that $K\subseteq\Phi(K)$ and also for each $y\in Y$ and each $n$ $\Phi(K)\times\{y\}\subseteq U$ for some $U\in\mathcal{U}_n$.

For each $n$ $\mathcal{W}_n=\{\Phi(K) : K\text{ is a compact subset of }X\}$ is a member of $\mathcal{G}_K$. Apply the property $SS_1^*(\mathcal{G}_K,\mathcal{G}_\Gamma)$ to $(\mathcal{W}_n)$ to obtain a sequence $(x_n)$ of elements of $X$ such that $\{St(x_n,\mathcal{W}_n) : n\in\mathbb{N}\}\in\mathcal{G}_\Gamma$ for $X$. Clearly for each $n$ $\mathcal{O}_n=\{O\subseteq Y : O\text{ is open in }Y\text{ and }\{x_n\}\times O\subseteq U\text{ for some }U\in\mathcal{U}_n\}\in\mathcal{O}_\cup$ for $Y$. Since $Y$ satisfies $SS_{comp}^*(\mathcal{O}_\cup,\mathcal{O})$, there exists a sequence $(C_n)$ of compact subsets of $Y$ such that $\{St(C_n,\mathcal{O}_n) : n\in\mathbb{N}\}$ is an $\omega$-cover in $Y$ (see Lemma~\ref{LN1}). For each $n$ let $F_n=\{x_n\}\times C_n$. It now remains to show that the sequence $(F_n)$ witnesses for $(\mathcal{U}_n)$ that $X\times Y$ satisfies $SS_{comp}^*(\mathcal{O}_\cup,\mathcal{O})$. Let $\langle x,y\rangle\in X\times Y$. Now $x\in St(x_n,\mathcal{W}_n)$ for all but finitely many $n$ because $\{St(x_n,\mathcal{W}_n) : n\in\mathbb{N}\}\in\mathcal{G}_\Gamma$. Also since $\{St(C_n,\mathcal{O}_n) : n\in\mathbb{N}\}$ is an $\omega$-cover, $y\in St(C_n,\mathcal{O}_n)$ for infinitely many $n$. Now choose a $n_0\in\mathbb{N}$ such that $x\in St(x_{n_0},\mathcal{W}_{n_0})$ and $y\in St(C_{n_0},\mathcal{O}_{n_0})$. We can find a compact set $K\subseteq X$, a set $O\in\mathcal{O}_{n_0}$ containing $y$ and a set $U_1\in\mathcal{U}_{n_0}$ such that $x, x_{n_0}\in\Phi(K)$, $O\cap C_{n_0}\neq\emptyset$ and $\{x_{n_0}\}\times O\subseteq U_1$. Clearly $U_1\cap F_{n_0}\neq\emptyset$. Also there is a set $U_2\in\mathcal{U}_{n_0}$ such that $\Phi(K)\times\{y\}\subseteq U_2$. The set $U=U_1\cup U_2\in\mathcal{U}_{n_0}$ has the property that $\langle x,y\rangle\in U$ and $U\cap F_{n_0}\neq\emptyset$. Observe that $\langle x,y\rangle\in St(F_{n_0},\mathcal{U}_{n_0})$ and accordingly $\{St(F_n,\mathcal{U}_n) : n\in\mathbb{N}\}
$ covers $X\times Y$.
\end{proof}

\begin{Th}
\label{TN25}
If $X$ satisfies $SS_1^*(\mathcal{G}_K,\mathcal{G}_\Gamma)$ and $Y$ satisfies $SS_{comp}^*(\mathcal{O}_\cup,\Gamma)$, then $X\times Y$ satisfies $SS_{comp}^*(\mathcal{O}_\cup,\Gamma)$.
\end{Th}

\begin{Th}
\label{TN5}
If $X$ satisfies $SS_1^*(\mathcal{G}_K,\mathcal{G}_\Gamma)$ and $Y$ satisfies $SS_{fin}^*(\mathcal{O}_\cup,\mathcal{O})$, then $X\times Y$ satisfies $SS_{fin}^*(\mathcal{O}_\cup,\mathcal{O})$.
\end{Th}

\begin{Cor}
\label{CN8}
If $X$ satisfies $SS_1^*(\mathcal{G}_K,\mathcal{G}_\Gamma)$ and $Y$ satisfies $SS_1^*(\mathcal{O}_\cup,\mathcal{O})$, then $X\times Y$ satisfies $SS_1^*(\mathcal{O}_\cup,\mathcal{O})$.
\end{Cor}

\begin{Th}
\label{TN26}
If $X$ satisfies $SS_1^*(\mathcal{G}_K,\mathcal{G}_\Gamma)$ and $Y$ satisfies $SS_{fin}^*(\mathcal{O}_\cup,\Gamma)$, then $X\times Y$ satisfies $SS_{fin}^*(\mathcal{O}_\cup,\Gamma)$.
\end{Th}

\begin{table}[H]
\begin{tabular}{|l|c|c|}
\hline
  Property & Property-preserving class\\
  \hline
  $S_{fin}^*(\mathcal{O}_\cup,\mathcal{O})$ & \multirow[c]{3}{*}{$S_1^*(\mathcal{G}_K,\mathcal{G}_\Gamma)$}\\
  \cline{1-1}
  $U_{fin}^*(\mathcal{O}_\cup,\Gamma)$ & \\
  \cline{1-1}
  $S_1^*(\mathcal{O}_\cup,\mathcal{O})$ & \\
  \hline
  $SS_{comp}^*(\mathcal{O}_\cup,\mathcal{O})$ & \multirow[c]{5}{*}{$SS_1^*(\mathcal{G}_K,\mathcal{G}_\Gamma)$}\\
  \cline{1-1}
  $SS_{comp}^*(\mathcal{O}_\cup,\Gamma)$ & \\
  \cline{1-1}
  $SS_{fin}^*(\mathcal{O}_\cup,\mathcal{O})$ & \\
  \cline{1-1}
  $SS_{fin}^*(\mathcal{O}_\cup,\Gamma)$ & \\
  \cline{1-1}
  $SS_1^*(\mathcal{O}_\cup,\mathcal{O})$ & \\
  \hline
\end{tabular}
\vskip0.1cm
\caption{Property-preserving class: using Alster covers}
\label{tab4}
\end{table}

\section{Results concerning cardinalities}
\subsection{Critical cardinalities $\mathfrak{d}$, $\mathfrak{b}$ and $\cov(\mathcal{M})$}

\begin{Th}[{\!\cite[Proposition 1.7]{SVM}}]
\label{TN7}
Every star-Lindel\"{o}f (respectively, strongly star-Lindel\"{o}f) space of cardinality less than $\mathfrak{d}$ is star-Menger (respectively, strongly star-Menger).
\end{Th}

\begin{Th}[{\!\cite[Corollary 3.10]{SSSP}}]
\label{TN54}
Every star-Lindel\"{o}f (respectively, strongly star-Lindel\"{o}f) space of cardinality less than $\mathfrak{b}$ is star-Hurewicz (respectively, strongly star-Hurewicz).
\end{Th}

\begin{Th}
\label{TN10}
Every star-Lindel\"{o}f (respectively, strongly star-Lindel\"{o}f) space of cardinality less than $\cov(\mathcal{M})$ is star-Rothberger (respectively, strongly star-Rothberger).
\end{Th}
\begin{proof}
Let $X$ be a star-Lindel\"{o}f space of cardinality less than $\cov(\mathcal{M})$. Choose a sequence $(\mathcal{U}_n)$ of open covers of $X$. For each $n$ we can find a countable subset $\mathcal{V}_n$ of $\mathcal{U}_n$ such that $X=St(\cup\mathcal{V}_n,\mathcal{U}_n)$. Define $\mathcal{V}_n=\{U_m^{(n)} : m\in\mathbb{N}\}$ for each $n$. For each $x\in X$ choose a function $f_x\in\mathbb{N}^\mathbb{N}$ such that $St(x,\mathcal{U}_n)\cap U_{f_x(n)}^{(n)}\neq\emptyset$ for all $n$. Since $\{f_x : x\in X\}$ is of cardinality less than $\cov(\mathcal{M})$, there is a $g\in\mathbb{N}^\mathbb{N}$ such that $\{f_x : x\in X\}$ can be guessed by $g$ i.e. $\{n\in\mathbb{N} : f_x(n)=g(n)\}$ is an infinite set for all $x\in X$. For each $n$ define $U_n=U_{g(n)}^{(n)}$. We claim that the sequence $(U_n)$ witnesses for $(\mathcal{U}_n)$ that $X$ is star-Rothberger. Let $x\in X$. Choose a $n_0\in\mathbb{N}$ such that $f_x(n_0)=g(n_0)$. From the construction of $f_x$ we obtain a $U\in\mathcal{U}_{n_0}$ such that $x\in U$ and $U\cap U_{f_x(n_0)}^{(n_0)}\neq\emptyset$. It follows that $U\cap U_{n_0}\neq\emptyset$ and hence $x\in St(U_{n_0},\mathcal{U}_{n_0})$. Thus $\{St(U_n,\mathcal{U}_n) : n\in\mathbb{N}\}$ covers $X$ and $X$ is star-Rothberger. Proof for the strongly star-Lindel\"{o}f case can be  similarly obtained.
\end{proof}
For investigations similar to Theorem~\ref{TN10}, see \cite{rosafil}.

\begin{table}[h!]
\centering
  \begin{tabular}{|l|c|l|c|}
  \hline
  Space with property & with cardinality $<$ & Equivalent to & Source \\
  \hline
  star-Lindel\"{o}f & \multirow[c]{2}{*}{$\mathfrak{d}$} & star-Menger & \multirow[c]{2}{*}{\cite{SVM}}\\
   \cline{1-1}\cline{3-3}
  strongly star-Lindel\"{o}f & & strongly star-Menger & \\
   \hline
  star-Lindel\"{o}f &  \multirow[c]{2}{*}{$\mathfrak{b}$} & star-Hurewicz & \multirow[c]{2}{*}{\cite{SSSP}}\\
  \cline{1-1}\cline{3-3}
  strongly star-Lindel\"{o}f & & strongly star-Hurewicz & \\
   \hline
  star-Lindel\"{o}f &  \multirow[c]{2}{*}{$\cov(\mathcal{M})$} & star-Rothberger &\\
   \cline{1-1}\cline{3-4}
  strongly star-Lindel\"{o}f & & strongly star-Rothberger &\\
   \hline
\end{tabular}
\vskip0.1cm
\caption{Property under cardinality bounds}
\end{table}

\begin{Th}[{\!\cite[Theorem 3.5]{SSSP}}]
\label{TN8}
Let $X$ be a regular space of the form $Y\cup Z$ with $Y\cap Z=\emptyset$, where $Y$ is a closed discrete set and $Z$ is a $\sigma$-compact subset of $X$. If $X$ is strongly star-Lindel\"{o}f, then $|Y|<\mathfrak{d}$ if and only if $X$ is strongly star-Menger.
\end{Th}

\begin{Th}[{\!\cite[Proposition 2]{SCPP}}]
\label{TN9}
The space $\Psi(\mathcal{A})$ is strongly star-Menger if and only if $|\mathcal{A}|<\mathfrak{d}$.
\end{Th}

\begin{Th}[{\!\cite[Theorem 3.12]{SSSP}}]
\label{TN53}
Let $X$ be a regular space of the form $Y\cup Z$ with $Y\cap Z=\emptyset$, where $Y$ is a closed discrete set and $Z$ is a $\sigma$-compact subset of $X$. If $X$ is strongly star-Lindel\"{o}f, then $|Y|<\mathfrak{b}$ if and only if $X$ is strongly star-Hurewicz.
\end{Th}

\begin{Th}[{\!\cite[Proposition 3]{SCPP}}]
\label{TN55}
The space $\Psi(\mathcal{A})$ is strongly star-Hurewicz if and only if $|\mathcal{A}|<\mathfrak{b}$.
\end{Th}

\begin{Def}
\label{DN2}
A subspace $Y$ of $X$ is said to be strongly star-Rothberger in $X$ if for every sequence $(\mathcal{U}_n)$ of open covers of $X$ there exists a sequence $(x_n)$ of elements of $X$ such that $\{St(x_n,\mathcal{U}_n) : n\in\mathbb{N}\}$ covers $Y$.
\end{Def}

It is immediate that every strongly star-Rothberger subspace of $X$ is strongly star-Rothberger in $X$, but the converse is not true (see Lemma~\ref{LN11}).
\begin{Lemma}
\label{LN11}
Assume $\omega_1<\cov(\mathcal{M})$. If $\Psi(\mathcal{A})$ is the Isbell-Mr\'{o}wka space with $|\mathcal{A}|=\omega_1$, then $\mathcal{A}$ is strongly star-Rothberger in $\Psi(\mathcal{A})$ but not strongly star-Rothberger.
\end{Lemma}
\begin{proof}
Since $\mathcal{A}$ is a discrete subspace of $\Psi(\mathcal{A})$ with $|\mathcal{A}|=\omega_1$, $\mathcal{A}$ is not strongly star-Rothberger.

Next we show that $\mathcal{A}$ is strongly star-Rothberger in $\Psi(\mathcal{A})$. Choose a sequence $(\mathcal{U}_n)$ of open covers of $\Psi(\mathcal{A})$. Without loss of generality assume that for each $n$ elements of $\mathcal{U}_n$ are basic sets: $$\mathcal{U}_n=\{U_n(A) : A\in\mathcal{A}\}\cup\{\{n\} : n\in\mathbb{N}\setminus\cup_{A\in\mathcal{A}} U_n(A)\}.$$ We can further assume that to each $A\in\mathcal{A}$ only one neighbourhood $U_n(A)\in\mathcal{U}_n$ is assigned. For each $A\in\mathcal{A}$ define a function $f_A\in\mathbb{N}^\mathbb{N}$ by $f_A(n)=\min\{m\in\mathbb{N} : m\in U_n(A)\}$. Since the collection $\{f_A : A\in\mathcal{A}\}$ is of cardinality less than $\cov(\mathcal{M})$, there exists a $g\in\mathbb{N}^\mathbb{N}$ such that the set $\{n\in\mathbb{N} : f_A(n)=g(n)\}$ is infinite for each $A\in\mathcal{A}$. Let $(x_n)$ be a sequence of elements of $\Psi(\mathcal{A})$ given by $x_n=g(n)$. We claim that the sequence $(x_n)$ witnesses for $(\mathcal{U}_n)$ that $\mathcal{A}$ is strongly star-Rothberger in $\Psi(\mathcal{A})$. Choose a $A\in\mathcal{A}$ and a $n_A\in\mathbb{N}$ such that $f_A(n_A)=g(n_A)$. From the construction of $f_A$ it follows that $x_{n_A}\in U_{n_A}(A)$ as $f_A(n_A)\in U_{n_A}(A)$ and $f_A(n_A)=g(n_A)$. Consequently $A\in St(x_{n_A},\mathcal{U}_{n_A})$ since $U_{n_A}\in\mathcal{U}_{n_A}$ with $A\in U_{n_A}$. Thus $\{St(x_n,\mathcal{U}_n) : n\in\mathbb{N}\}$ covers $\mathcal{A}$ and hence $\mathcal{A}$ is strongly star-Rothberger in $\Psi(\mathcal{A})$.
\end{proof}

\begin{Lemma}
\label{LN12}
If $X=\cup_{k\in\mathbb{N}}X_k$ with each $X_k$ is strongly star-Rothberger in $X$, then $X$ is strongly star-Rothberger.
\end{Lemma}

We give an alternative proof of the next result.
\begin{Th}[{\!\cite[Proposition 4]{SCPP}}]
\label{TN56}
If $|\mathcal{A}|<\cov(\mathcal{M})$, then $\Psi(\mathcal{A})$ is strongly star-Rothberger.
\end{Th}
\begin{proof}
By Lemma~\ref{LN11}, $\mathcal{A}$ is strongly star-Rothberger in $\Psi(A)$. Also $\mathbb{N}$ is strongly star-Rothberger in $\Psi(\mathcal{A})$ since $\mathbb{N}$ is strongly star-Rothberger. Thus $\Psi(\mathcal{A})$ is strongly star-Rothberger by Lemma~\ref{LN12}.
\end{proof}

\begin{Th}[{\!\cite[Proposition 2.3]{rosa22}}]
\label{TN18}
If a Lindel\"{o}f space $X$ is union of less than $\mathfrak{d}$ star-Hurewicz spaces, then $X$ is star-Menger.
\end{Th}

\begin{Th}[{\!\cite[Proposition 2.4]{rosa22}}]
\label{TN21}
If a Lindel\"{o}f space $X$ is union of less than $\mathfrak{b}$ star-Hurewicz spaces, then $X$ is star-Hurewicz.
\end{Th}

\begin{Th}
\label{TN15}
If a Lindel\"{o}f space $X$ is union of less than $\cov(\mathcal{M})$ $S_1^*(\mathcal{O},\Gamma)$ spaces, then $X$ is star-Rothberger.
\end{Th}
\begin{proof}
Let $X=\cup_{\alpha<\kappa}X_\alpha$ with $\kappa<\cov(\mathcal{M})$, where each $X_\alpha$ is a $S_1^*(\mathcal{O},\Gamma)$ space. Consider a sequence $(\mathcal{U}_n)$ of open covers of $X$. Since $X$ is Lindel\"{o}f, for each $n$ we can write $\mathcal{U}_n=\{U_m^{(n)} : m\in\mathbb{N}\}$. Applying $S_1^*(\mathcal{O},\Gamma)$ to $(\mathcal{U}_n)$, for each $\alpha<\kappa$ we obtain a sequence $(V_n^{(\alpha)})$ such that for each $n$ $V_n^{(\alpha)}\in\mathcal{U}_n$ and $\{St(V_n^{(\alpha)},\mathcal{U}_n) : n\in\mathbb{N}\}$ is a $\gamma$-cover of $X_\alpha$ by open sets in $X$. For each $\alpha<\kappa$ choose a function $f_\alpha\in\mathbb{N}^\mathbb{N}$ such that $V_n^{(\alpha)}=U_{f_\alpha(n)}^{(n)}$ for all $n$. Since the collection $\{f_\alpha : \alpha<\kappa\}$ has cardinality less than $\cov(\mathcal{M})$, there is a function $g\in\mathbb{N}^\mathbb{N}$ such that $\{n\in\mathbb{N} : f_\alpha(n)=g(n)\}$ is infinite for each $\alpha<\kappa$. For each $n$ define $U_n=U_{g(n)}^{(n)}$. We now show that the sequence $(U_n)$ witnesses for $(\mathcal{U}_n)$ that $X$ is star-Rothberger. Let $x\in X$ and choose a $\beta<\kappa$ such that $x\in X_\beta$. Now choose a $n_0\in\mathbb{N}$ such that $x\in St(V_n^{(\beta)},\mathcal{U}_n)$ for all $n\geq n_0$. Since the set  $\{n\in\mathbb{N} : f_\beta(n)=g(n)\}$ is infinite, there is a $k\in\mathbb{N}$ such that $k\geq n_0$ and $f_\beta(k)=g(k)$. It follows that $x\in St(V_k^{(\beta)},\mathcal{U}_k)$ and so $x\in St(U_{f_\beta(k)}^{(k)},\mathcal{U}_k)$ i.e. $x\in St(U_{g(k)}^{(k)},\mathcal{U}_k)$. Which in turn implies that $x\in St(U_k,\mathcal{U}_k)$ and consequently $\{St(U_n,\mathcal{U}_n) : n\in\mathbb{N}\}$ covers $X$.
\end{proof}

\begin{Th}[{\!\cite[Proposition 2.14]{rosa22}}]
\label{TN20}
If a star-Lindel\"{o}f space $X$ is union of less than $\mathfrak{d}$ Hurewicz spaces, then $X$ is star-Menger.
\end{Th}

\begin{Th}[{\!\cite[Proposition 2.15]{rosa22}}]
\label{TN23}
If a star-Lindel\"{o}f space $X$ is union of less than $\mathfrak{b}$ Hurewicz spaces, then $X$ is star-Hurewicz.
\end{Th}

A space $X$ is said to be a star-$\epsilon$-space if for every open cover $\mathcal{U}$ of $X$ there exists a countable set $\mathcal{V}\subseteq\mathcal{U}$ such that $\{St(V,\mathcal{U}) : V\in\mathcal{V}\}$ is an $\omega$-cover of $X$.
\begin{Th}
\label{TN17}
If a star-$\epsilon$-space $X$ is union of less than $\cov(\mathcal{M})$ $S_1(\Omega,\Gamma)$ spaces, then $X$ is star-Rothberger.
\end{Th}
\begin{proof}
Let $X=\cup_{\alpha<\kappa}X_\alpha$ with $\kappa<\cov(\mathcal{M})$, where each $X_\alpha$ is a $S_1(\Omega,\Gamma)$ space. Consider a sequence $(\mathcal{U}_n)$ of open covers of $X$. Since $X$ is a star-$\epsilon$-space, for each $n$ there is a countable subset $\mathcal{V}_n=\{U_m^{(n)} : m\in\mathbb{N}\}$ of $\mathcal{U}_n$ such that $\{St(U_m^{(n)},\mathcal{U}_n) : m\in\mathbb{N}\}$ is an $\omega$-cover of $X$. For each $n$ let $\mathcal{W}_n=\{St(U_m^{(n)},\mathcal{U}_n) : m\in\mathbb{N}\}$. Applying the hypothesis to $(\mathcal{W}_n)$, for each $\alpha<\kappa$ we obtain a sequence $(V_n^{(\alpha)})$ such that for each $n$ $V_n^{(\alpha)}\in\mathcal{W}_n$ and $\{V_n^{(\alpha)} : n\in\mathbb{N}\}$ is a $\gamma$-cover of $X_\alpha$ by open sets in $X$. For each $\alpha<\kappa$ choose a function $f_\alpha\in\mathbb{N}^\mathbb{N}$ such that $V_n^{(\alpha)}=St(U_{f_\alpha(n)}^{(n)},\mathcal{U}_n)$ for all $n$. Since the collection $\{f_\alpha : \alpha<\kappa\}$ has cardinality less than $\cov(\mathcal{M})$, there exists a function $g\in\mathbb{N}^\mathbb{N}$ such that $\{n\in\mathbb{N} : f_\alpha(n)=g(n)\}$ is infinite for each $\alpha<\kappa$. Define $U_n=U_{g(n)}^{(n)}$ for each $n$. We now show that the sequence $(U_n)$ witnesses for $(\mathcal{U}_n)$ that $X$ is star-Rothberger. Now choose a $x\in X$ so that $x\in X_\beta$ for some $\beta<\kappa$. Next find a $n_0\in\mathbb{N}$ such that $x\in V_n^{(\beta)}$ for all $n\geq n_0$. Since the set  $\{n\in\mathbb{N} : f_\beta(n)=g(n)\}$ is infinite, choose a $k\in\mathbb{N}$ such that $k\geq n_0$ and $f_\beta(k)=g(k)$. It follows that $x\in V_k^{(\beta)}$ and so $x\in St(U_{f_\beta(k)}^{(k)},\mathcal{U}_k)$. Subsequently $x\in St(U_{g(k)}^{(k)},\mathcal{U}_k)$ i.e. $x\in St(U_k,\mathcal{U}_k)$. Thus $\{St(U_n,\mathcal{U}_n) : n\in\mathbb{N}\}$ is a cover of $X$. \end{proof}

\begin{Th}[{\!\cite[Proposition 2.11]{rosa22}}]
\label{TN19}
If a strongly star-Lindel\"{o}f space $X$ is union of less than $\mathfrak{d}$ Hurewicz spaces, then $X$ is strongly star-Menger.
\end{Th}

\begin{Th}[{\!\cite[Proposition 2.12]{rosa22}}]
\label{TN22}
If a strongly star-Lindel\"{o}f space $X$ is union of less than $\mathfrak{b}$ Hurewicz spaces, then $X$ is strongly star-Hurewicz.
\end{Th}

A space $X$ is said to be a strongly star-$\epsilon$-space if for every open cover $\mathcal{U}$ of $X$ there exists a countable set $A\subseteq X$ such that $\{St(x,\mathcal{U}) : x\in A\}$ is an $\omega$-cover of $X$.
\begin{Th}
\label{TN16}
If a strongly star-$\epsilon$-space $X$ is union of less than $\cov(\mathcal{M})$ $S_1(\Omega,\Gamma)$ spaces, then $X$ is strongly star-Rothberger.
\end{Th}
\begin{proof}
Let $X=\cup_{\alpha<\kappa}X_\alpha$ with $\kappa<\cov(\mathcal{M})$, where each $X_\alpha$ is a $S_1(\Omega,\Gamma)$ space. Consider a sequence $(\mathcal{U}_n)$ of open covers of $X$. Since $X$ is a strongly star-$\epsilon$-space, for each $n$ there exists a countable subset $A_n=\{x_m^{(n)} : m\in\mathbb{N}\}$ of $X$ such that $\{St(x_m^{(n)},\mathcal{U}_n) : m\in\mathbb{N}\}$ is an $\omega$-cover of $X$. For each $n$ let $\mathcal{W}_n=\{St(x_m^{(n)},\mathcal{U}_n) : m\in\mathbb{N}\}$. Apply the hypothesis to obtain for each $\alpha<\kappa$ a sequence $(V_n^{(\alpha)})$ such that for each $n$ $V_n^{(\alpha)}\in\mathcal{W}_n$ and $\{V_n^{(\alpha)} : n\in\mathbb{N}\}$ is a $\gamma$-cover of $X_\alpha$ by open sets in $X$. For each $\alpha<\kappa$ we choose a function $f_\alpha\in\mathbb{N}^\mathbb{N}$ such that $V_n^{(\alpha)}=St(x_{f_\alpha(n)}^{(n)},\mathcal{U}_n)$ for all $n$. Since the collection $\{f_\alpha : \alpha<\kappa\}$ has cardinality less than $\cov(\mathcal{M})$, there exists a function $g\in\mathbb{N}^\mathbb{N}$ such that $\{n\in\mathbb{N} : f_\alpha(n)=g(n)\}$ is infinite for each $\alpha<\kappa$. For each $n$ define $x_n=x_{g(n)}^{(n)}$. We show that the sequence $(x_n)$ witnesses for $(\mathcal{U}_n)$ that $X$ is strongly star-Rothberger. Choose a $x\in X$ such that $x\in X_\beta$ for some $\beta<\kappa$. Now there is a $n_0\in\mathbb{N}$ such that $x\in V_n^{(\beta)}$ for all $n\geq n_0$. Since the set  $\{n\in\mathbb{N} : f_\beta(n)=g(n)\}$ is infinite, there is a $k\in\mathbb{N}$ such that $k\geq n_0$ and $f_\beta(k)=g(k)$. It follows that $x\in V_k^{(\beta)}$ and hence $x\in St(x_{f_\beta(k)}^{(k)},\mathcal{U}_k)$. Thus $x\in St(x_{g(k)}^{(k)},\mathcal{U}_k)$ i.e. $x\in St(x_k,\mathcal{U}_k)$. Clearly $\{St(x_n,\mathcal{U}_n) : n\in\mathbb{N}\}$ covers $X$.
\end{proof}

\begin{Th}[{\!\cite[Proposition 2.7]{rosa22}}]
\label{TN12}
If a Lindel\"{o}f space $X$ is union of less than $\mathfrak{b}$ star-Menger spaces, then $X$ is star-Menger.
\end{Th}

\begin{Th}[{\!\cite[Proposition 2.16]{rosa22}}]
\label{TN14}
If a star-Lindel\"{o}f space $X$ is union of less than $\mathfrak{b}$ Menger spaces, then $X$ is star-Menger.
\end{Th}

\begin{Th}[{\!\cite[Proposition 2.13]{rosa22}}]
\label{TN13}
If a strongly star-Lindel\"{o}f space $X$ is union of less than $\mathfrak{b}$ Menger spaces, then $X$ is strongly star-Menger.
\end{Th}

\subsection{The extent}
Recall that for a space $X$, \[e(X)=\sup\{|Y| : Y\;\text{is a closed and discrete subspace of}\; X\}\] is said to be the extent of $X$.
\begin{Lemma}[{\!\cite[Lemma 2.6]{rssM}}]
\label{LN8}
The extent of a $T_1$ space $X$ and its Alexandroff duplicate $A(X)$ are equal. 
\end{Lemma}

\begin{Th}[{\!\cite[Theorem 2.5]{rsM-II}}]
\label{TN43}
Every $T_1$ star-Menger space $X$ with countable extent has star-Menger Alexandroff duplicate $A(X)$.
\end{Th}

\begin{Th}
\label{T1}
Every $T_1$ star-Rothberger space $X$ with countable extent has star-Rothberger Alexandroff duplicate $A(X)$.
\end{Th}
\begin{proof}
Let $(\mathcal{U}_n)$ be a sequence of open covers of $A(X)$. For each $n$ and each $x\in X$ consider the open set $U_x^{(n)}=(V_x^{(n)}\times\{0,1\})\setminus\{\langle x,1\rangle\}$ containing $\langle x,0\rangle$ in $A(X)$, where $V_x^{(n)}$ is an open set in $X$ containing $x$ and $U_x^{(n)}\subseteq U$ for some $U\in\mathcal{U}_n$. Choose a sequence $(\mathcal{W}_n)$ of open covers of $X$, where $\mathcal{W}_n=\{V_x^{(n)} : x\in X\}$. Next choose two infinite disjoint subsets $N_1$ and $N_2$ of $\mathbb{N}$ such that $\mathbb{N}=N_1\cup N_2$. Since $X$ is star-Rothberger, there is a sequence $(V_n : n\in N_1)$ such that for each $n\in N_1$, $V_n\in\mathcal{W}_n$ and $\{St(V_n,\mathcal{W}_n) : n\in N_1\}$ covers $X$. For each $n\in N_1$ choose a $U_n\in\mathcal{U}_n$ such that $V_n\times\{0\}\subseteq U_n$. Clearly $X\times\{0\}\subseteq\cup_{n\in N_1}St(U_n,\mathcal{U}_n)$. Observe that $A(X)\setminus\cup_{n\in N_1}St(U_n,\mathcal{U}_n)$ is a closed discrete subset of $A(X)$. By Lemma~\ref{LN8}, $A(X)\setminus\cup_{n\in N_1}St(U_n,\mathcal{U}_n)$ is countable and we enumerate it bijectively as $\{y_n : n\in N_2\}$. For each $n\in N_2$ choose a $U_n\in\mathcal{U}_n$ such that $y_n\in U_n$. Thus the sequence $(U_n)$ witnesses for $(\mathcal{U}_n)$ that $A(X)$ is star-Rothberger.
\end{proof}

\begin{Th}[{\!\cite[Theorem 2.12]{RSKM}}]
\label{TN41}
Every $T_1$ star-K-Menger space $X$ with countable extent has star-K-Menger Alexandroff duplicate $A(X)$.
\end{Th}

\begin{Th}[{\!\cite[Theorem 2.7]{rssM}}]
\label{TN42}
Every $T_1$ strongly star-Menger space $X$ with countable extent has strongly star-Menger Alexandroff duplicate $A(X)$.
\end{Th}

\begin{Th}
\label{T1}
Every $T_1$ strongly star-Rothberger space $X$ with countable extent has strongly star-Rothberger Alexandroff duplicate $A(X)$.
\end{Th}

We do not know whether similar result holds for the star variants of the Hurewicz property.
\begin{Prob}\rm
\label{QN1}
Let $X$ be a $T_1$ star-Hurewicz (respectively, star-K-Hurewicz, strongly star-Hurewicz) space with countable extent. Is the Alexandroff duplicate $A(X)$ of $X$ star-Hurewicz (respectively, star-K-Hurewicz, strongly star-Hurewicz)?
\end{Prob}

\begin{Th}
\label{TN27}
If the Alexandroff duplicate $A(X)$ of a $T_1$ space $X$ is star-Lindel\"{o}f, then $X$ has countable extent.
\end{Th}
\begin{proof}
Assume on the contrary that $e(X)\geq\omega_1$. Accordingly there exists a closed discrete subset $D$ of $X$ with $|D|\geq\omega_1$. Observe that $D\times\{1\}$ is a clopen subset of $A(X)$ and since each point of $D\times\{1\}$ is isolated, it is also discrete. Since the star-Lindel\"{o}fness is preserved under clopen subsets, $A(X)$ is not star-Lindel\"{o}f. Which is not possible.
\end{proof}

As a consequence of the preceding observation we obtain  \cite[Theorem 2.6]{rsM-II}, \cite[Theorem 2.13]{RSKM}, \cite[Theorem 2.8]{rssM}, \cite[Theorem 2.5]{RSH} and \cite[Theorem 2.4]{RSSH}.

\begin{Cor}
\label{T2}
Every $T_1$ space $X$ has countable extent, provided the Alexandroff duplicate $A(X)$ satisfies any one of the following properties.
\begin{multicols}{2}
\begin{enumerate}[wide=0pt,label={\upshape(\arabic*)},leftmargin=*]
  \item star-Menger {\rm\cite{rsM-II}}.
  \item star-Hurewicz {\rm\cite{RSH}}.
  \item star-Rothberger.
  \item star-K-Menger {\rm\cite{RSKM}}.
  \item star-K-Hurewicz.
  \item strongly star-Menger {\rm\cite{rssM}}.
  \item strongly star-Hurewicz {\rm\cite{RSSH}}.
  \item strongly star-Rothberger.
\end{enumerate}
\end{multicols}
\end{Cor}

Also see \cite{rsM,rsM-II,rssM} for detailed investigations on the extent of star-Menger (star-Hurewicz, strongly star-Menger, strongly star-Hurewicz, strongly star-Rothberger) spaces.
The following examples illustrate the behaviour of the extent of star-K-Menger and star-K-Hurewicz spaces. If $\kappa$ is any infinite cardinal, then by \cite[Theorem 1]{HWWE}, there exists a Tychonoff space $X$ with $e(X)\geq\kappa$. The next example shows that such a space is indeed star-K-Hurewicz (hence star-K-Menger).
\begin{Ex}[{\!\cite[Theorem 1]{HWWE}}]\rm
\label{E6}
For any infinite cardinal $\kappa$, there exists a Tychonoff star-K-Hurewicz (hence star-K-Menger) strongly star-Lindel\"{o}f space $X(\kappa)$ such that $e(X(\kappa))\geq\kappa$.
\end{Ex}
\begin{proof}
For each $\alpha<\kappa$, define $f_\alpha\in\{0,1\}^\kappa$ by
\begin{equation*}
f_\alpha(\beta)=
\begin{cases}
1 & \text{if~} \beta=\alpha\\
0 & \text{otherwise}.
\end{cases}
\end{equation*}
Choose $D=\{f_\alpha : \alpha<\kappa\}$ and $$X(\kappa)=(\{0,1\}^\kappa\times(\kappa^++1))
\setminus((\{0,1\}^\kappa\setminus D)\times\{\kappa^+\})$$ as a subspace of the product space $\{0,1\}^\kappa\times(\kappa^++1)$. In \cite[Theorem 1]{HWWE}, Matveev showed that $X(\kappa)$ is a Tychonoff strongly star-Lindel\"{o}f space and $D\times\{\kappa^+\}$ is a closed and discrete subset of it. Hence $e(X(\kappa))\geq\kappa$.

It now suffices to show that $X(\kappa)$ is K-starcompact. Let $\mathcal{U}$ be an open cover of $X(\kappa)$. Choose a $\beta<\kappa^+$ such that $$D\times\{\kappa^+\}\subseteq St(\{0,1\}^\kappa\times[0,\beta],\mathcal{U})$$ (see the proof of \cite[Theorem 1]{HWWE}). Since $\{0,1\}^\kappa\times\kappa^+$ is countably compact, it is K-starcompact. Subsequently there is a compact subset $K$ of $X(\kappa)$ such that $$\{0,1\}^\kappa\times\kappa^+\subseteq St(K,\mathcal{U}).$$ Thus, like $X(\kappa)=(D\times\{\kappa^+\})\cup(\{0,1\}^\kappa\times\kappa^+)$, the set $(\{0,1\}^\kappa\times[0,\beta])\cup K$ witnesses that $X(\kappa)$ is K-starcompact. This concludes the proof.
\end{proof}

\begin{Ex}[{\!\cite[Lemma 2.3]{sKH}}]\rm
\label{E5}
For any infinite cardinal $\kappa>\omega$, there exists a Tychonoff star-K-Hurewicz (hence star-K-Menger) space $Y(\kappa)$ such that $e(Y(\kappa))\geq\kappa$ which is not strongly star-Lindel\"{o}f.
\end{Ex}
\begin{proof}
Consider $$Y(\kappa)=(\beta D\times[0,\kappa^+])\setminus((\beta D\setminus D)\times\{\kappa^+\})$$ as a subspace of $\beta D\times[0,\kappa^+]$, where $D=\{d_\alpha : \alpha<\kappa\}$ is the discrete space of cardinality $\kappa$. Then $Y(\kappa)$ is a Tychonoff star-K-Hurewicz space (see \cite[Lemma 2.3]{sKH}). It is easy to see that $D\times\{\kappa^+\}$ is a discrete and closed subset of $Y(\kappa)$. Thus $e(Y(\kappa))\geq\kappa$.

Next we claim that $Y(\kappa)$ is not strongly star-Lindel\"{o}f. Suppose that $Y(\kappa)$ is strongly star-Lindel\"{o}f. Obviously $$\mathcal{U}=\{\beta D\times[0,\kappa^+)\}\cup\{\{d_\alpha\}\times[0,\kappa^+] : \alpha<\kappa\}$$ is an open cover of $Y(\kappa)$. Then there exists a countable subset $A$ of $Y(\kappa)$ such that $Y(\kappa)=St(A,\mathcal{U})$. Since $A$ is countable, we get a $\gamma<\kappa$ with $$A\cap(\{d_\alpha\}\times[0,\kappa^+])=\emptyset$$ for all $\alpha>\gamma$. Choose $\gamma<\alpha_0<\kappa$. Then $$\langle d_{\alpha_0},\kappa^+\rangle\notin St(A,\mathcal{U})$$ because $\{d_{\alpha_0}\}\times[0,\kappa^+]$ is the only member of $\mathcal{U}$ containing the point $\langle d_{\alpha_0},\kappa^+\rangle$, which is a contradiction. Hence $Y(\kappa)$ is not strongly star-Lindel\"{o}f.
\end{proof}

A space $X$ is said to be metacompact (respectively, subparacompact) \cite{Burke} if every open cover of it has a point-finite open refinement (respectively, $\sigma$-discrete closed refinement). Observe that the spaces $X(\kappa)$ and $Y(\kappa)$ in the preceding examples contain a non-compact countably compact closed subspace which is homeomorphic to $\kappa^+$. Thus the spaces $X(\kappa)$ and $Y(\kappa)$ as considered above fail to be metacompact and subparacompact as well. The following problems can be raised.

\begin{Prob}\rm
\label{Q1}
Can the extent of a metacompact (or, subparacompact) star-K-Hurewicz space be arbitrarily large?
\end{Prob}

\begin{Prob}\rm
\label{QN2}
Can the extent of a metacompact (or, subparacompact) star-K-Menger space be arbitrarily large?
\end{Prob}

Next we give an affirmative answer to the above problems.
\begin{Ex}[{\!\cite[Example 3.4]{SVM}}]\rm
\label{E7}
For any infinite cardinal $\kappa$, there exists a Hausdorff (non-regular) metacompact subparacompact star-K-Hurewicz (hence star-K-Menger) space $Z(\kappa)$ such that $e(Z(\kappa))\geq\kappa$. Moreover $Z(\kappa)$ is not strongly star-Lindel\"{o}f for $\kappa>\omega$.
\end{Ex}
\begin{proof}
Let $aD=D\cup\{d\}$ be the one point compactification of the discrete space $D=\{d_\alpha : \alpha<\kappa\}$ of cardinality $\kappa$. We substitute the local base of the point $\langle d,\omega\rangle$ by the family $$\mathcal{B}=\{U\setminus(D\times\{\omega\}) : \langle d,\omega\rangle\in U\;\text{and}\; U\;\text{is an open set in}\; aD\times[0,\omega]\}$$ in the product space $aD\times[0,\omega]$ and let $Z(\kappa)$ be the space obtained by such substitution. Note that $Z(\kappa)$ is a Hausdorff (non-regular) metacompact subparacompact space with $e(Z(\kappa))\geq\kappa$ and also it is not strongly star-Lindel\"{o}f for $\kappa>\omega$ (see \cite[Example 3.4]{SVM}).

To show that $Z(\kappa)$ is star-K-Hurewicz, it suffices to show that $Z(\kappa)$ is K-starcompact. Let $\mathcal{U}$ be an open cover of $Z(\kappa)$. We pick a $V\in\mathcal{U}$ with $\langle d,\omega\rangle\in V$. Then there exists a $U\setminus(D\times\{\omega\})\in\mathcal{B}$ such that $U\setminus(D\times\{\omega\})\subseteq V$. We claim that $U\setminus(D\times\{\omega\})$ is a compact subset of $Z(\kappa)$. Let $\mathcal{V}$ be a cover of $U\setminus(D\times\{\omega\})$ by open sets in $Z(\kappa)$. Choose a $W\in\mathcal{V}$ and a $O\setminus(D\times\{\omega\})\in\mathcal{B}$ such that $$\langle d,\omega\rangle\in W\;\text{and}\;O\setminus(D\times\{\omega\})\subseteq W.$$ It can be easily concluded that a finite subset of $\mathcal{V}$ covers $U\setminus(D\times\{\omega\})$ because $O\setminus(D\times\{\omega\})$ contains all but finitely many points of $U\setminus(D\times\{\omega\})$. Thus $U\setminus(D\times\{\omega\})$ is a compact subset of $Z(\kappa)$. One can readily observe that $U\setminus(D\times\{\omega\})$ does not contain only the points of  $D\times\{\omega\}$ and the points of $$\{\langle d_\alpha, m\rangle : \alpha<\alpha_0\;\text{and}\;m<m_0\}$$ for some finite $\alpha_0<\kappa$ and for some $m_0\in\omega$. Now $$K=(U\setminus(D\times\{\omega\}))\cup\{\langle d_\alpha, m\rangle : \alpha<\alpha_0\;\text{and}\;m<m_0\}$$ guarantees for $\mathcal{U}$ that $Z(\kappa)$ is K-starcompact since $K$ intersects every member of $\mathcal{U}$.
\end{proof}

\section{Diagrams for star selection principles}
\begin{Def}
\label{DN1}
An open cover $\mathcal{U}=\{U_\alpha : \alpha<\kappa\}$ of a space $X$ is said to be a modified large cover if for each $\beta<\kappa$, $\{U_\alpha : \beta\leq\alpha<\kappa\}$ is also a cover of $X$. The collection of all modified large covers of $X$ is denoted by $\Lambda^*$.
\end{Def}

\begin{Lemma}
\label{LN7}
For a space $X$ the following assertions hold.
\begin{enumerate}[wide=0pt,label={\upshape(\arabic*)},leftmargin=*,
ref={\theLemma(\arabic*)}]
  \item\label{LN701} $S_{fin}(\mathcal{O},\mathcal{O})=S_{fin}(\mathcal{O},\Lambda^*)$.
  \item\label{LN1302} $SS_{fin}^*(\mathcal{O},\mathcal{O})=SS_{fin}^*(\mathcal{O},\Lambda^*)$.
   \item\label{LN1301} $SS_{comp}^*(\mathcal{O},\mathcal{O})=SS_{comp}^*(\mathcal{O},\Lambda^*)$.
  \item\label{LN702} $S_{fin}^*(\mathcal{O},\mathcal{O})=S_{fin}^*(\mathcal{O},\Lambda^*)$.
\end{enumerate}
\end{Lemma}

\begin{Lemma}
\label{LN6}
For a space $X$ the following assertions hold.
\begin{enumerate}[wide=0pt,label={\upshape(\arabic*)},leftmargin=*]
  \item\label{LN601} $S_1(\mathcal{O},\mathcal{O})=S_1(\mathcal{O},\Lambda^*)$.
  \item\label{LN602} $SS_1^*(\mathcal{O},\mathcal{O})=SS_1^*(\mathcal{O},\Lambda^*)$.
  \item\label{LN602} $S_1^*(\mathcal{O},\mathcal{O})=S_1^*(\mathcal{O},\Lambda^*)$.
\end{enumerate}
\end{Lemma}

The results obtained so far can be summarized into the following implication diagrams Figures~\ref{dig2}, \ref{dig3} and \ref{dig4}.
\begin{figure}[h!]
\begin{adjustbox}{max width=\textwidth,max height=\textheight,keepaspectratio,center}
\begin{tikzcd}[column sep=0.5ex,row sep=6ex,arrows={crossing over}]
U_{fin}^*(\Omega,\Gamma)&&
S_{fin}^*(\Omega,\Gamma)\arrow[ll]\arrow[rr]&&
S_{fin}^*(\Omega,\Omega)\arrow[rr]&&
S_{fin}^*(\Omega,\Lambda)\arrow[rr]&&
S_{fin}^*(\Omega,\Lambda^*)\arrow[rr,equal]&&
S_{fin}^*(\Omega,\mathcal{O})&
\\
&S_1^*(\Omega,\Gamma)\arrow[rr]\arrow[ru,equal]&&
S_1^*(\Omega,\Omega)\arrow[rr]\arrow[ru]&&
S_1^*(\Omega,\Lambda)\arrow[rr]\arrow[ru]&&
S_1^*(\Omega,\Lambda^*)\arrow[rr,equal]\arrow[ru]&&
S_1^*(\Omega,\mathcal{O})\arrow[ru]&&
\\
U_{fin}^*(\Lambda,\Gamma)\arrow[uu]&&
S_{fin}^*(\Lambda,\Gamma)\arrow[ll]\arrow[uu]\arrow[rr]&&
S_{fin}^*(\Lambda,\Omega)\arrow[uu]\arrow[rr]&&
S_{fin}^*(\Lambda,\Lambda)\arrow[uu]\arrow[rr]&&
S_{fin}^*(\Lambda,\Lambda^*)\arrow[uu]\arrow[rr,equal]&&
S_{fin}^*(\Lambda,\mathcal{O})\arrow[uu]&
\\
&S_1^*(\Lambda,\Gamma)\arrow[rr]\arrow[uu]\arrow[ru,equal]&&
S_1^*(\Lambda,\Omega)\arrow[rr]\arrow[uu]\arrow[ru]&&
S_1^*(\Lambda,\Lambda)\arrow[rr]\arrow[uu]\arrow[ru]&&
S_1^*(\Lambda,\Lambda^*)\arrow[rr,equal]\arrow[uu]\arrow[ru]&&
S_1^*(\Lambda,\mathcal{O})\arrow[uu]\arrow[ru]&&
\\
U_{fin}^*(\Lambda^*,\Gamma)\arrow[uu]&&
S_{fin}^*(\Lambda^*,\Gamma)\arrow[ll]\arrow[uu]\arrow[rr]&&
S_{fin}^*(\Lambda^*,\Omega)\arrow[uu]\arrow[rr]&&
S_{fin}^*(\Lambda^*,\Lambda)\arrow[uu]\arrow[rr]&&
S_{fin}^*(\Lambda^*,\Lambda^*)\arrow[uu]\arrow[rr,equal]&&
S_{fin}^*(\Lambda^*,\mathcal{O})\arrow[uu]&
\\
&S_1^*(\Lambda^*,\Gamma)\arrow[rr]\arrow[uu]\arrow[ru,equal]&&
S_1^*(\Lambda^*,\Omega)\arrow[rr]\arrow[uu]\arrow[ru]&&
S_1^*(\Lambda^*,\Lambda)\arrow[rr]\arrow[uu]\arrow[ru]&&
S_1^*(\Lambda^*,\Lambda^*)\arrow[rr,equal]\arrow[uu]\arrow[ru]&&
S_1^*(\Lambda^*,\mathcal{O})\arrow[uu]\arrow[ru]&&
\\
U_{fin}^*(\mathcal{O},\Gamma)\arrow[uu]&&
S_{fin}^*(\mathcal{O},\Gamma)\arrow[ll]\arrow[uu]\arrow[rr]&&
S_{fin}^*(\mathcal{O},\Omega)\arrow[uu]\arrow[rr]&&
S_{fin}^*(\mathcal{O},\Lambda)\arrow[uu]\arrow[rr]&&
S_{fin}^*(\mathcal{O},\Lambda^*)\arrow[uu]\arrow[rr,equal]&&
S_{fin}^*(\mathcal{O},\mathcal{O})\arrow[uu]&
\\
&S_1^*(\mathcal{O},\Gamma)\arrow[rr]\arrow[uu]\arrow[ru,equal]&&
S_1^*(\mathcal{O},\Omega)\arrow[rr]\arrow[uu]\arrow[ru]&&
S_1^*(\mathcal{O},\Lambda)\arrow[rr]\arrow[uu]\arrow[ru]&&
S_1^*(\mathcal{O},\Lambda^*)\arrow[rr,equal]\arrow[uu]\arrow[ru]&&
S_1^*(\mathcal{O},\mathcal{O})\arrow[uu]\arrow[ru]&&
\end{tikzcd}
\end{adjustbox}
\caption{Diagram for star-selection principles}
\label{dig2}
\end{figure}

\begin{figure}[h!]
\begin{adjustbox}{max width=\textwidth,max height=\textheight,keepaspectratio,center}
\begin{tikzcd}[column sep=0.5ex,row sep=6ex,arrows={crossing over}]
&SS_{fin}^*(\Omega,\Gamma)\arrow[rr]&&
SS_{fin}^*(\Omega,\Omega)\arrow[rr]&&
SS_{fin}^*(\Omega,\Lambda)\arrow[rr]&&
SS_{fin}^*(\Omega,\Lambda^*)\arrow[rr,equal]&&
SS_{fin}^*(\Omega,\mathcal{O})&
\\
SS_1^*(\Omega,\Gamma)\arrow[rr]\arrow[ru,equal]&&
SS_1^*(\Omega,\Omega)\arrow[rr]\arrow[ru]&&
SS_1^*(\Omega,\Lambda)\arrow[rr]\arrow[ru]&&
SS_1^*(\Omega,\Lambda^*)\arrow[rr,equal]\arrow[ru]&&
SS_1^*(\Omega,\mathcal{O})\arrow[ru]&&
\\
&SS_{fin}^*(\Lambda,\Gamma)\arrow[uu]\arrow[rr]&&
SS_{fin}^*(\Lambda,\Omega)\arrow[uu]\arrow[rr]&&
SS_{fin}^*(\Lambda,\Lambda)\arrow[uu]\arrow[rr]&&
SS_{fin}^*(\Lambda,\Lambda^*)\arrow[uu]\arrow[rr,equal]&&
SS_{fin}^*(\Lambda,\mathcal{O})\arrow[uu]&
\\
SS_1^*(\Lambda,\Gamma)\arrow[rr]\arrow[uu]\arrow[ru,equal]&&
SS_1^*(\Lambda,\Omega)\arrow[rr]\arrow[uu]\arrow[ru]&&
SS_1^*(\Lambda,\Lambda)\arrow[rr]\arrow[uu]\arrow[ru]&&
SS_1^*(\Lambda,\Lambda^*)\arrow[rr,equal]\arrow[uu]\arrow[ru]&&
SS_1^*(\Lambda,\mathcal{O})\arrow[uu]\arrow[ru]&&
\\
&SS_{fin}^*(\Lambda^*,\Gamma)\arrow[uu]\arrow[rr]&&
SS_{fin}^*(\Lambda^*,\Omega)\arrow[uu]\arrow[rr]&&
SS_{fin}^*(\Lambda^*,\Lambda)\arrow[uu]\arrow[rr]&&
SS_{fin}^*(\Lambda^*,\Lambda^*)\arrow[uu]\arrow[rr,equal]&&
SS_{fin}^*(\Lambda^*,\mathcal{O})\arrow[uu]&
\\
SS_1^*(\Lambda^*,\Gamma)\arrow[rr]\arrow[uu]\arrow[ru,equal]&&
SS_1^*(\Lambda^*,\Omega)\arrow[rr]\arrow[uu]\arrow[ru]&&
SS_1^*(\Lambda^*,\Lambda)\arrow[rr]\arrow[uu]\arrow[ru]&&
SS_1^*(\Lambda^*,\Lambda^*)\arrow[rr,equal]\arrow[uu]\arrow[ru]&&
SS_1^*(\Lambda^*,\mathcal{O})\arrow[uu]\arrow[ru]&&
\\
&SS_{fin}^*(\mathcal{O},\Gamma)\arrow[uu]\arrow[rr]&&
SS_{fin}^*(\mathcal{O},\Omega)\arrow[uu]\arrow[rr]&&
SS_{fin}^*(\mathcal{O},\Lambda)\arrow[uu]\arrow[rr]&&
SS_{fin}^*(\mathcal{O},\Lambda^*)\arrow[uu]\arrow[rr,equal]&&
SS_{fin}^*(\mathcal{O},\mathcal{O})\arrow[uu]&
\\
SS_1^*(\mathcal{O},\Gamma)\arrow[rr]\arrow[uu]\arrow[ru,equal]&&
SS_1^*(\mathcal{O},\Omega)\arrow[rr]\arrow[uu]\arrow[ru]&&
SS_1^*(\mathcal{O},\Lambda)\arrow[rr]\arrow[uu]\arrow[ru]&&
SS_1^*(\mathcal{O},\Lambda^*)\arrow[rr,equal]\arrow[uu]\arrow[ru]&&
SS_1^*(\mathcal{O},\mathcal{O})\arrow[uu]\arrow[ru]&&
\end{tikzcd}
\end{adjustbox}
\caption{Diagram for strongly star-selection principles}
\label{dig3}
\end{figure}

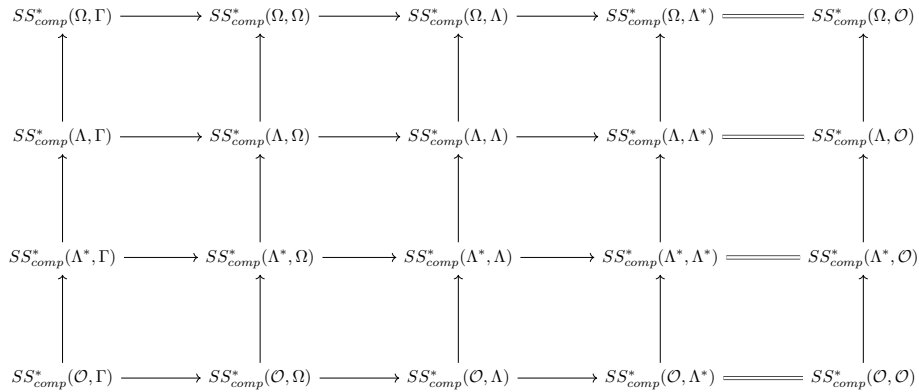
\begin{figure}[h!]
\begin{adjustbox}{max width=\textwidth,max height=\textheight,keepaspectratio,center}
\begin{tikzcd}[column sep=5ex,row sep=6ex,arrows={crossing over}]
SS_{comp}^*(\Omega,\Gamma)\arrow[rr]&&
SS_{comp}^*(\Omega,\Omega)\arrow[rr]&&
SS_{comp}^*(\Omega,\Lambda)\arrow[rr]&&
SS_{comp}^*(\Omega,\Lambda^*)\arrow[rr,equal]&&
SS_{comp}^*(\Omega,\mathcal{O})
\\
&&&&&&&&&
\\
SS_{comp}^*(\Lambda,\Gamma)\arrow[uu]\arrow[rr]&&
SS_{comp}^*(\Lambda,\Omega)\arrow[uu]\arrow[rr]&&
SS_{comp}^*(\Lambda,\Lambda)\arrow[uu]\arrow[rr]&&
SS_{comp}^*(\Lambda,\Lambda^*)\arrow[uu]\arrow[rr,equal]&&
SS_{comp}^*(\Lambda,\mathcal{O})\arrow[uu]
\\
&&&&&&&&
\\
SS_{comp}^*(\Lambda^*,\Gamma)\arrow[uu]\arrow[rr]&&
SS_{comp}^*(\Lambda^*,\Omega)\arrow[uu]\arrow[rr]&&
SS_{comp}^*(\Lambda^*,\Lambda)\arrow[uu]\arrow[rr]&&
SS_{comp}^*(\Lambda^*,\Lambda^*)\arrow[uu]\arrow[rr,equal]&&
SS_{comp}^*(\Lambda^*,\mathcal{O})\arrow[uu]
\\
&&&&&&&&
\\
SS_{comp}^*(\mathcal{O},\Gamma)\arrow[uu]\arrow[rr]&&
SS_{comp}^*(\mathcal{O},\Omega)\arrow[uu]\arrow[rr]&&
SS_{comp}^*(\mathcal{O},\Lambda)\arrow[uu]\arrow[rr]&&
SS_{comp}^*(\mathcal{O},\Lambda^*)\arrow[uu]\arrow[rr,equal]&&
SS_{comp}^*(\mathcal{O},\mathcal{O})\arrow[uu]
\end{tikzcd}
\end{adjustbox}
\caption{Diagram for star-K-selection principles}
\label{dig4}
\end{figure}

\vskip 4cm
\noindent{\bf Acknowledgement:} The authors are thankful to the referee(s) for numerous useful comments and suggestions that considerably improved the presentation of the paper.

\end{document}